%% file: adaptive_quantization.tex
\documentclass[final,onefignum,onetabnum,reqno]{siamart171218}

\usepackage{amsmath, amsfonts, amssymb, graphicx, cite, epsfig,epstopdf,color,soul}

\usepackage{url}
\usepackage[titletoc,page]{appendix}
\usepackage{enumerate}
\usepackage[shortlabels]{enumitem}
\usepackage{empheq}

\usepackage{subcaption}
\usepackage{ifthen}
\usepackage{multirow, multicol}
\usepackage{float}
\usepackage{subeqnarray, array}
\usepackage{xcolor}
\usepackage[font=small]{caption}
\usepackage{tikz}

\input{header}

\title{Fast Convergence Rates of Distributed Subgradient Methods with Adaptive Quantization}
\author{{Thinh T. Doan$^{\dagger,\star}$, Siva Theja Maguluri$^{\star}$, Justin Romberg$^{\dagger}$}\\
$^\dagger$ School of Electrical and Computer Engineering\\
$^\star$ H. Milton Stewart School of Industrial and Systems Engineering\\
Georgia Institute of Technology, GA, 30332, USA\\ 
{\tt\small \{thinhdoan,\,siva.theja\}@gatech.edu, jrom@ece.gatech.edu.}} 
\date{}
\begin{document}
\maketitle

\begin{abstract}
We study distributed optimization problems over a network when the communication between the nodes is constrained, and so information that is exchanged between the nodes must be quantized.  Recent advances using the distributed gradient algorithm with a quantization scheme at a fixed resolution have established convergence, but at rates significantly slower than when the communications are unquantized.

In this paper, we introduce a novel quantization method, which we refer to as  adaptive quantization, that allows us to match the convergence rates under perfect communications.  Our approach adjusts the quantization scheme used by each node as the algorithm progresses: as we approach the solution, we become more certain about where the state variables are localized, and adapt the quantizer codebook accordingly.  

We bound the convergence rates of the proposed method as a function of the communication bandwidth, the underlying network topology, and structural properties of the constituent objective functions.  In particular, we show that if the objective functions are convex or strongly convex, then using adaptive quantization does not affect the rate of convergence of the distributed subgradient methods when the communications are quantized, except for a constant that depends on the resolution of the quantizer. To the best of our knowledge, the rates achieved in this paper are better than any existing work in the literature for distributed gradient methods under finite communication bandwidths.  We also provide numerical simulations that compare convergence properties of the distributed gradient methods with and without quantization for solving distributed regression problems for both quadratic and absolute loss functions.     
\end{abstract}

\input{intro}
%\input{notation}
\input{alg}

\input{analysis}

\input{simulations}

\section{Concluding Remarks}
In this paper, we consider distributed optimization over networks of nodes under finite bandwidths, and so information exchanged across the network must be quantized. For solving such problems, we consider distributed subgradient methods under quantization. Our main contribution is to propose a novel adaptive quantization, which quantizes the nodes' estimates based on the progress of the algorithm. Under this adaptive quantization, we show that the rates of convergence of {\sf DSG} are unaffected by communication constraints. A natural question from this work is to ask whether the proposed adaptive quantization can be extended to study other distributed algorithms under finite bandwidths, such as, distributed primal-dual, ADMM, dual averaging, and mirror-descent. Indeed, it is not obvious whether we can meet the conditions presented in Section \ref{subsec:alg}. We leave such an interesting question for our future research.

\bibliographystyle{IEEEtran}
\bibliography{refs}
%\newpage
\input{results_proofs}

\end{document}

%% file: header.tex
\newcommand{\Rset}{\mathbb{R}}

\newcommand{\Ecal}{{\cal E}}

\newcommand{\Gcal}{{\cal G}}

\newcommand{\Ncal}{{\cal N}}
\newcommand{\Ocal}{{\cal O}}

\newcommand{\Qcal}{{\cal Q}}
\newcommand{\Rcal}{{\cal R}}

\newcommand{\Vcal}{{\cal V}}

\newcommand{\Xcal}{{\cal X}}

\newcommand{\Abf}{{\bf A}}

\newcommand{\Gbf}{{\bf G}}

\newcommand{\Ibf}{{\bf I}}

\newcommand{\Pbf}{{\bf P}}
\newcommand{\Qbf}{{\bf Q}}

\newcommand{\Vbf}{{\bf V}}
\newcommand{\Wbf}{{\bf W}}
\newcommand{\Xbf}{{\bf X}}
\newcommand{\Ybf}{{\bf Y}}

\newcommand{\abf}{{\bf a}}

\newcommand{\gbf}{{\bf g}}

\newcommand{\lbf}{{\boldsymbol{\ell}}}

\newcommand{\pbf}{{\bf p}}
\newcommand{\qbf}{{\bf q}}
\newcommand{\rbf}{{\bf r}}

\newcommand{\ubf}{{\bf u}}
\newcommand{\vbf}{{\bf v}}
\newcommand{\wbf}{{\bf w}}
\newcommand{\xbf}{{\bf x}}
\newcommand{\ybf}{{\bf y}}
\newcommand{\zbf}{{\bf z}}

\newcommand{\xibarbf}{{\bar{\xibf}}}
\newcommand{\taubf}{{\boldsymbol{\tau}}}

\newcommand{\qbarbf}{{\bar{\qbf}}}
\newcommand{\xbarbf}{{\bar{\xbf}}}
\newcommand{\vbarbf}{{\bar{\vbf}}}

\newcommand{\1}{{\mathbf{1}}}

\newcommand{\xibf}{{\boldsymbol{\xi}}}
\newcommand{\Xibf}{{\boldsymbol{\Xi}}}

\newtheorem{assump}{Assumption}

\newcommand{\redline}{\raisebox{2pt}{\tikz{\draw[-,red,solid,line width = 1.5pt](0,0) -- (6mm,0);}}}
\newcommand{\blueline}{\raisebox{2pt}{\tikz{\draw[-,blue,solid,line width = 1.5pt](0,0) -- (6mm,0);}}}
\newcommand{\greenline}{\raisebox{2pt}{\tikz{\draw[-,green,solid,line width = 1.5pt](0,0) -- (6mm,0);}}}

\newcommand{\blackline}{\raisebox{2pt}{\tikz{\draw[-,black,solid,line width = 1.5pt](0,0) -- (6mm,0);}}}

%% file: intro.tex
%!TEX root = adaptive_quantization.tex

\section{Introduction}

We consider a distributed subgradient algorithm for solving optimization problems of the form 
\begin{align}
    \underset{\xbf\in\Xcal}{\text{minimize }} f(\xbf)
    \triangleq \sum_{i=1}^n f_i(\xbf).\label{prob:obj}
\end{align}
Each functional $f_i$ in the sum above is associated with a computational node, and these nodes have connected into a network specified by a graph.  Each node has knowledge of only their function $f_i$, and so they must work together, communicating only with their neighbors on the graph, to find a minimizer of \eqref{prob:obj}.  The distributed subgradient ({\sf DSG}) method, described in full in Section~\ref{subsec:alg} below, is a popular approach for solving this type of problem.  In {\sf DSG}, each node keeps a local estimate of the decision variable $\xbf$ and iterates by communicating the local state to its neighbors on the graph, averaging the estimates it received from its neighbors, then taking a gradient step. The convergence properties of this algorithm are well understood, see \cite{NedicO2009} along with the other references in Section~\ref{subsec:related}, and essentially match the rates of standard centralized subgradient methods with a constant factor that depends on the connectivity of the graph.

In this paper, we take a step towards understanding how imperfect communication affects the convergence of {\sf DSG}.  In particular, we derive convergence rates for a modified version of the {\sf DSG} algorithm when the communications between the nodes are {\em quantized}, modeling scenarios in which the bandwidth available for communication is often limited.  Unlike previous works \cite{DoanMR2018b,ReisizadehMHP2018}, we show if the quantization intervals are adjusted as the algorithm approaches a solution, then we can match the convergence rates of unquantized {\sf DSG} within a constant that depends on the number of bits used for quantization.  We call this novel coding method {\em adaptive quantization}, as it changes at every iteration based on the stepsize being used for the subgradient descent.

\subsection{Main Contribution}\label{subsec:Contribution}
We first propose a modified version of {\sf DSG}, which directly takes into account the quantized error at every iteration. Second, we design a novel adaptive quantization method, where the nodes quantize their values based on the progress of their updates.  The entire method is summarized in Algorithm~\ref{alg:QDSG} and described fully in Section~\ref{sec:QDSG} below.

Our analytical contribution is to show that the convergence rates of {\sf DSG} are unaffected when the communications use adaptive quantization, except for a factor which captures the size of communication bandwidth.

Our analysis treats both the cases where the objective functions are convex and strongly convex.  When the $f_i$ are convex, we show that at iteration $k$, each node $i$ has an estimate $\zbf_i(k)$ that obeys
\[
    f(\zbf_i(k)) - f^* ~\lesssim~
    \frac{\Delta^2}{(1-\sigma_2)^2}\cdot \frac{\ln k}{\sqrt{k}},
\]
where $1-\sigma_2$ is the spectral gap that quantifies the connectivity of the underlying network and $\Delta$ is the resolution of the quantizer. When the $f_i$ are strongly convex, we derive a rate on the convergence of the $\zbf_i$ to the unique solution $\xbf^*$,
\[
    \|\zbf_i(k)-\xbf^*\|^2 ~\lesssim~
    \frac{\Delta^2}{(1-\sigma_2)^2}\cdot \frac{\ln k}{k}.
\]
These rates match those for the standard, unquantized verion of {\sf DSG}  \cite{NedicOR2018}.

The numerical results in Section~\ref{sec:simulation} show that for stylized problems, both smooth and not smooth, quantizing to $8$ bits is essentially the same as communicating real numbers, while using  $5$ or $6$ bits results in only a modest increase in the number iterations required to converge to a specified tolerance.

\subsection{Related Work}
\label{subsec:related}
The {\sf DSG} algorithms for solving problem \eqref{prob:obj} have a long history, probably first studied in \cite{Tsitsiklis1986} and recently received a wide attention; see for example,  \cite{NedicO2009,ShiLWY2015,GuLi2017, NedicOS2017,LorenzoS2016, XiMXAK2018, XiXK2018} and the recent survey paper \cite{NedicOR2018}. Convergence results of {\sf DSG} have been explicitly studied in this literature; however, they are mostly established under a critical assumption on the perfect communication between nodes. Such assumption is not often held in practice, therefore, there is a necessity to study the performance of {\sf DSG} under imperfect communication. In particular, our focus in this paper is to study the convergence rates of this method when information exchanged between the nodes are quantized, modeling the practical applications with finite communication bandwidth.      

Distributed algorithms with random (dithered) quantization have been considered in \cite{AysalCR2008} for solving network consensus problems, a special case of the problem considered in this paper. On the other hand, different variants of distributed gradient methods under quantized communication have been studied in \cite{JueyouGZX2016,  NedicOOT2008b,PuZJ2017, DoanMR2018a, DoanMR2018b, ReisizadehMHP2018,YiH2014}. In \cite{JueyouGZX2016, NedicOOT2008b} the authors only show the convergence to a neighborhood around the optimal of the problem, while an exact convergence has been studied in \cite{PuZJ2017,DoanMR2018a}; however, a condition on the growing communication bandwidth is assumed in the latter work. To remove such strong condition, the authors in \cite{ReisizadehMHP2018,DoanMR2018b} show the asymptotic convergence of {\sf DSG} methods under random quantization using only finite bandwidth. In particular, in \cite{ReisizadehMHP2018} the authors provide a rate in expectation  $\mathcal{O}(1/k^{(1-\gamma)/2})$, for some $\gamma\in(0,1)$, when the problem objectives are smooth and strongly convex.  On the other hand, in \cite{DoanMR2018b} we study distributed subgradient methods for nonsmooth problems and analyze their convergence rates by utilizing techniques from stochastic approximation approach. Specifically, such algorithms asymptotically converge to the optimal value in expectation at a rate $\mathcal{O}(ln(k)\,/\,k^{1/4})$ and  $\mathcal{O}(ln(k)\,/\,k^{1/3})$ for convex and strongly convex functions, respectively. The rates established in these two papers, however, are sub-optimal and much slower than the ones in this paper as stated in Section \ref{subsec:Contribution}. 

The adaptive quantization studied in this paper seems to share some similarity with the so-called ``zoom in" and ``zoom out" quantization to study the stability of linear systems \cite{BrockettL2000}. Moreover, this ``zoom in" and ``zoom out" quantization has also been applied in distributed optimization with finite bandwidths, where an asymptotic convergence to a solution has been derived in \cite{YiH2014}. However, there is a lack of understanding how fast the algorithm converges, which is one of the main focus of this paper.

We also want to note some related work \cite{Alistarh_NIPS2017,Alistarh_NIPS2018,Stich_NIPS2018} and the references therein, in which distributed stochastic gradient with quantization under master/worker models is considered. Such models consider a special star graph communication structure, while consensus-based gradient methods are designed for any network topology. In general, these two approaches are fundamentally different, therefore, the results studied in master/worker models cannot be extended to cover the problem considered in this paper. Moreover, the quantized communication constraint studied in this paper is one example of imperfect exchange of information between nodes.  Another example of imperfect exchange is latency in the communications.  Convergence rates of {\sf DSG} optimization methods in the presence of communication delays have been studied in \cite{ WuYLYS2018, TianSDS2018,DoanBS2018, DoanBS2017}. 

Finally, there are some related methods based on primal-dual approach for solving problem \eqref{prob:obj}, such as, the accelerated primal-dual methods \cite{LanLZ2017, Scaman2018}, the \textit{alternating direction method of multipliers (\sf ADMM)} \cite{WeiO2012,BoydPCPE2011,ShiLYWY2014, ChangHW2015,ChangHLW2016}, and the \textit{distributed dual methods} (mirror descent/dual averaging) \cite{DuchiAW2012, TsianosLR2012, DoanBNB2019}. Our focus in this paper will be on {\sf DSG} algorithms, as they are both simple and have convergence guarantees that are as strong or stronger than those for dual methods.

\subsection{Notation}\label{subsec:notation}
We introduce here a set of definitions and notation that is used throughout this paper. We use boldface to denote vectors in $\Rset^d$ to distinguish them from scalars in $\Rset$. Given a collection of vectors $\xbf_1,\ldots,\xbf_n$ in $\Rset^d$, we denote by $\Xbf$ a matrix in $\Rset^{n\times d}$, whose $i$-th row is $\xbf_i^T$. We then denote by $\|\xbf\|$ and $\|\Xbf\|$ the Euclidean norm  and the Frobenius norm of the vector $\xbf$ and the meatrix $\Xbf$, respectively.  We use $\1\in\Rset^{d}$ to denote the the vector whose entries are all $1$ and $\Ibf\in\Rset^{n\times n}$ to denote the identity matrix. Given a closed convex set $\Xcal$, we denote by $[\xbf]_{\Xcal}$ the projection of $\xbf$ to $\Xcal$. 

Given a nonsmooth convex function $f:\Rset^{d}\rightarrow\Rset$, we denote by $\partial f(\xbf)$ its subdifferential  at $x$, which is defined as the set of subgradients of $f$ at $\xbf$, i.e., $\partial f(\xbf) \triangleq \{\gbf\in\Rset^{d}\,|\, f(\ybf) \geq f(\xbf) + \gbf^T(\ybf-\xbf), \; \forall \ybf\in\Rset^d \}$.  Since $f$ is convex, $\partial f(\cdot)$  is nonempty. A function $f$ is said to be $L$-Lipschitz continuous if 
\begin{align}
|\,f(\xbf)-f(\ybf)\,| \leq L\|\xbf-\ybf\|,\quad \forall\; \xbf,\ybf\in \Rset^d. \label{notation:Lipschitz}
\end{align}   
Note that the $L$-Lipschitz continuity of $f$ is equivalent to the subgradients of $f$ being uniformly bounded by $L$ \cite{ShalevShwartz2012}. A function $f$ is said to be $\mu$-strongly convex if  $f$ satisfies
\begin{align}
f(\ybf) - f(\xbf) -\gbf(\xbf)^{T}(\ybf-\xbf)\geq \frac{\mu}{2}\|\ybf-\xbf\|^2\qquad  \forall \xbf,\ybf. \label{notation:sc}
\end{align}

Note that since the set $\Xcal$ is compact, there exists a point $\xbf^*$ which solves Problem \eqref{prob:obj}. However, $\xbf^*$ may not be unique. We will use $\Xcal^*$ to denote the set of optimal solutions to Problem \eqref{prob:obj}. Given a solution $\xbf^*\in\Xcal^*$ we denote $f^* =\sum_{i=1}^n f_i(\xbf^*)$. Also, due to the compactness of $\Xcal$, the subgradients of $f_i$ are uniformly bounded in $\Xcal$. We state this observation formally in the following proposition. 
\begin{proposition}\label{prop:bounded_subg}
There exists a positive constant $L_i$, for all $i\in\Vcal$, such that the $2$-norm of subgradients $\gbf_i(\cdot)$ of $f_i$ are uniformly bounded by $L_i$ in $\Xcal$, i.e., the following condition holds 
\begin{align}
\|\gbf_i(\xbf)\| \leq L_i,\qquad \text{for all }\xbf\in\Xcal.\label{prop:bounded_subg_ineq} 
\end{align}
\end{proposition}

% Finally, when the quantized interval depends on time, i.e., $[c(k),d(k)]$, we denote by $\Qcal_k$ the time-varying quantization associate with $[c(k),d(k)]$. In the next section, we propose such a time-varying quantization, which is based on the progress of the algorithm. 

% The remainder of this paper is organized as follows. A list of notation and some preliminary results are provided in Section \ref{sec:notation}. Our proposed distributed consensus-based subgradient method under adaptive quantization is presented in Section \ref{sec:QDSG}, while their convergence analysis is given in Section \ref{sec:analysis}. We delay all the proofs of the results in Section \ref{sec:analysis} to Section \ref{sec:proofs}. In Section \ref{sec:simulation} we provide some numerical experiments to illustrate the effectiveness of our results. Finally, some extensions of our main results are given in the Appendix.  

%% file: alg.tex
%!TEX root = adaptive_quantization.tex

\section{Distributed Subgradient Methods}\label{sec:QDSG}
For solving problem \eqref{prob:obj}, we are interested in {\sf DSG} methods \cite{NedicOP2010}, where each node $i$ maintains its own version of the decision variables $\xbf_i\in\Rset^{d}$; the goal is to have all the $\xbf_i$ converge to $\xbf^*$, a solution of problem \eqref{prob:obj}. Each node is only allowed to interact with its neighbors that are directly  connected to it through a connected and undirected graph $\Gcal = (\Vcal,\Ecal)$, where $\Vcal = \{1,\ldots,n\}$ and  $\Ecal = (\Vcal\times\Vcal)$ are the vertex and edge sets, respectively. Each node $i$ then iteratively updates $\xbf_i$ as 
\begin{align}
\xbf_i(k+1) = \left[\sum_{j\in\Ncal_i} a_{ij}\xbf_j(k) \; - \; \alpha(k)  \gbf_i(\xbf_i(k))\right]_{\Xcal},\label{distributed:DSG}
\end{align}  
where $\alpha(k)$ is some sequence of stepsizes, $\gbf_i(\xbf_i(k))\in\partial f_i(\xbf_i(k))$, and $\Ncal:= \{ j \in \Vcal\; |\; (i, j) \in\Ecal\}$ is the set of node $i'$s neighbors. The $a_{ij}$ above are positive weights that can be non-zero if there is an edge between nodes $i$ and $j$, and otherwise can be assigned by node $i$.  We will collect these weights into a $n\times n$ matrix $\Abf$, and assume it meets the following conditions throughout.
\begin{assump}\label{assump:doub_stoch}
The matrix $\Abf$, whose $(i,j)$-th entries are $a_{ij}$, is doubly stochastic, i.e., $\sum_{i=1}^n a_{ij} =1$ for all $j$ and $\sum_{j=1}^n a_{ij} = 1$ for all $i$. Moreover, $\Abf$ is irreducible and aperiodic. Finally, the weights $a_{ij} > 0$ if and only if $(i, j) \in \Ecal$ otherwise $a_{ij} = 0$.
\end{assump} 
This assumption also implies that $\Abf$ has a largest singular value of $1$, and its other singular values are strictly less than $1$; see for example, the Perron-Frobenius theorem \cite{HJ1985}.  We denote by $\sigma_2\in(0,1)$ the second largest singular value of $\Abf$, which is a key quantity in the analysis of the mixing time of a Markov chain with transition probabilities given by $\Abf$.

% by the Courant-Fisher theorem \cite{HJ1985} gives
% \begin{align}
% \left\|\Abf\left(\Ibf-\frac{1}{n}\1\1^T\right)\right\| \leq \sigma_2\left\|\Ibf-\frac{1}{n}\1\1^T\right\|.\label{const:sigma_2}
% \end{align}
  
\subsection{Adaptive Quantization}\label{subsec:adaptive_quant}
Each iteration in \eqref{distributed:DSG} requires every node to communicate its estimate of the decision variables to its neighbors.  Theorems~\ref{thm_convex:rate} and \ref{thm_sconvex:rate} below study the convergence rate of a modified version of \eqref{distributed:DSG} when these communications are quantized using our proposed adaptive quantization method. We first present in this section some fundamentals of quantized communication.

To explain the main idea of our approach, we start with the uniform quantization method to quantize a single real number $x\in[\ell,u]$. In particular, we divide the interval into $B$ bins whose end points are denoted by $\tau_i$, $\ell = \tau_1\leq\tau_2\leq\ldots\leq\tau_B = u$.   We assume that the points $\tau_i$ are uniformly spaced with distance $\Delta$, i.e., $\Delta = \tau_{i+1}-\tau_i = (u-\ell)\,/\,(B-1)$ for all $i=0,\ldots,B-1$.  Thus $b = \lceil\log_2(B)\rceil$ bits can be used to index the $\{\tau_i\}$.  

Next, given a value $x\in[\ell,u]$ we denote by $q = \Qcal(x)$ its quantized value where  
\begin{align}
\Qcal(x) \triangleq \min_{\ell} |\tau_{\ell}-x|.\label{notation:Q_operator}  
\end{align}
If $\tau_{i}$ and $\tau_{i+1}$ achieve the minimal value in Eq.\ \eqref{notation:Q_error}, then without loss of generality we set $\Qcal(x) = \tau_i$. Also, by Eq.\ \eqref{notation:Q_operator} the quantized error is given as 
\begin{align}
| x - \Qcal(x) |\leq \Delta = \frac{u-\ell}{B-1}.\label{notation:Q_error}
\end{align}
When the quantized interval depends on time, i.e., $[\ell(k),u(k)]$, we denote by $\Qcal_k(x)$ the uniform quantization of $x$ over the time-varying interval $[\ell(k),u(k)]$ at time $k$. We can see from Eq. \eqref{notation:Q_error} that when $B$ is fixed the quantized error depends directly on the size of the interval $[\ell(k),u(k)]$. The main idea of our approach is to refine this interval at every time step so that this interval is shrinking, which implies that the quantized error decays to zero. We will refer to this scheme as  adaptive quantization, where a distributed implementation will be proposed in the next section.  

Moreover, since the constraint set $\Xcal$ is compact, it is contained in a rectangular set
\begin{align*}
\Xcal\subset\Rcal \triangleq [\lbf,\ubf] = [\ell^1,u^1]\times\ldots\times[\ell^d,u^d],
\end{align*}
for some $\{(\ell^i,u^i)\}$. We quantize a $\xbf\in\Rcal$ by applying the procedure above independently on each component, reusing the notation $\Qcal(\xbf)$ to indicate that \eqref{notation:Q_operator} is applied component wise. Thus, in this case the total number of bits required to quantize the whole vector is $b\times d$. Moreover, the constants in the convergence results presented below will be a function of the interval lengths $u^i-\ell^i$ for each coordinate $i$ of $\Xcal$.

Finally, as will be seen in Eq.\ \eqref{distributed:QDSG} in the next section, to implement the adaptive quantization method and update its estimate each node $i\in\Vcal$ has to apply the following two steps of its encoding/decoding scheme in each iteration $k\geq0$. 
\begin{enumerate}[leftmargin = 1.1cm]
\item Node $i$ computes the quantized interval $\Rcal_i(k)$ to find its quantized value $\qbf_i(k) = \Qcal_{k}(\xbf_i(k))$ by using the uniform quantization over this interval. As mentioned, this interval has to be computed in such a way that the quantized error $\Delta_i(k) = \xbf_i(k) - \qbf_i(k)$ decays to zero.
\item We note that for each iteration $k\geq0$ node $i$ only receives a sequence of $b$ bits $\qbf_j^b(k)$, e.g. $\qbf_j^b(k) = 01010010$ a sequence of $8$ bits, representing the value of $\qbf_j(k)$ over $\Rcal_j(k)$. Thus, node $i$ has to decode $\qbf_j^b(k)$ to recover $\qbf_j(k)$. This step can be done if node $i$ knows $\Rcal_j(k)$ for $j\in\Ncal_i$.  
\end{enumerate}

\subsection{Distributed Subgradient Methods under Adaptive Quantization}\label{subsec:alg}
Our focus in this paper is to study the impact of quantized communication between the nodes on the performance of {\sf DSG}. In particular, at any iteration $k\geq0$ the nodes are only allowed to send and receive the quantized values of their estimates to their neighboring nodes. Due to the quantization, we first modify the update in Eq.\ \eqref{distributed:DSG} to take into account the quantized error. That is, each node $i$, for all $i\in\Vcal$, now considers the following update
\begin{align}
\xbf_i(k+1) = \left[\underbrace{\xbf_i(k) - \qbf_i(k)}_{\text{``quantization error"}} + \sum_{j\in\Ncal_i}a_{ij}\qbf_j(k) \; - \; \alpha(k)  \gbf_i(\xbf_i(k))\right]_{\Xcal},\label{distributed:QDSG}
\end{align}
where $\qbf_i(k) = \Qcal_k(\xbf_i(k))$, for all $i\in\Vcal$, is the quantized value of $\xbf_i(k)$ over some interval $\Rcal_i(k) \triangleq [\lbf_i(k),\ubf_i(k)]$ for each step $k\geq0$. The update \eqref{distributed:QDSG} has a simple interpretation as follows. At time $k\geq0$, each node $i$ first obtains the quantized value $\qbf_i(k)$ of its value $\xbf_i(k)$. Each node $i$ then formulates the weighted average of its quantized value $\qbf_i(k)$ and the quantized values $\qbf_j(k)$ received from its neighbors $j\in\Ncal_j$, with the goal of seeking a consensus on their estimates. In addition, each node also introduces its quantized error into its own update, with the goal of eliminating the bias due to the quantized error over the network.  Each node then moves along the subgradients of its respective objective function to update its estimates, pushing the consensus point toward the optimal set $\Xcal^*$. The distributed subgradient algorithm under random quantization is formally stated in Algorithm \ref{alg:QDSG}.   

\begin{algorithm}[h]
\caption{Distributed Subgradient Algorithm Under Adaptive Quantization}
\vspace{0.2cm}
\begin{enumerate}[leftmargin = 4mm]
\item \textbf{Initialize}: Let $L = \underset{i\in\Vcal}{\sum}\; L_i,$ $\gamma = \frac{96 + 48L}{1-\sigma_2}.$

Each node $i\in\Vcal$ initializes 
\begin{itemize}
\item[a.] Divide $[\lbf,\ubf]$ (that contains $\Xcal$) into $B^d$ rectangular bins uniformly componentwise as described above, i.e., $\lbf = \taubf_1\leq \ldots\leq \taubf_{B^d} = \ubf$ .\vspace{0.1cm}
\item[b.] Set $\xbf_i(0) =\qbf_i(0) = \taubf_{m}$ for some arbitrarily index $m\in[1,B^d]$. Compute $\qbf_i^b(0)$ using the quantization scheme $\Qcal_{0}(\xbf_i(0))$ over the set $[\lbf,\ubf]$.
\item[c.] A sequence of positive and nonincreasing step sizes $\{\alpha(k)\}.$ 
\end{itemize}
\vspace{0.2cm}
\item \textbf{Iteration}:  For $k=0,1,\ldots,$ node $i\in\Vcal$ implements
\begin{itemize}
\item[a.] Send $\qbf_i^{b}(k)$ to node $j\in\Ncal_i$
\item[b.] Receive $\qbf_j^b(k)$ from node $j\in\Ncal_i$ 
\begin{itemize}
\item If $k=0$ then $\Rcal_j(0) = [\lbf,\ubf]$, otherwise  
\[
\Rcal_j(k) \triangleq \left[\qbf_j(k-1) - \frac{\gamma}{2}\alpha(k-1)\1,\,\qbf_j(k-1) + \frac{\gamma}{2}\alpha(k-1)\1\right]
\]
\item Recover $\qbf_j(k)$ from $\qbf_j^{b}(k)$ by using uniform quantization over $\Rcal_j(k)$
\end{itemize}
\item[c.] Use $\qbf_j(k)$ and update
\begin{align*}
\xbf_i(k+1) = \left[\xbf_i(k) - \qbf_i(k) + \sum_{j\in\Ncal_i}a_{ij}\qbf_j(k) \; - \; \alpha(k)  \gbf_i(\xbf_i(k))\right]_{\Xcal}.
\end{align*}
\item[d.] Compute $\qbf_i^b(k+1)$ and $\qbf_i(k+1)$ by using $\Qcal_{k+1}(\xbf_i(k+1))$ over the interval 
\[
\Rcal_i(k+1) \triangleq \left[\qbf_i(k) - \frac{\gamma}{2}\alpha(k)\1,\,\qbf_i(k) + \frac{\gamma}{2}\alpha(k)\1\right]
\]
\item[e.] Update the output 
\begin{align}
\zbf_i(k) = \frac{\sum_{t=0}^{k}\alpha(t)\xbf_i(t)}{\sum_{t=0}^{k}\alpha(t)}\cdot\label{alg_QDSG:zi}
\end{align}
\end{itemize}
\end{enumerate}
\label{alg:QDSG}
\end{algorithm}

Steps (a)-(d) in Algorithm \ref{alg:QDSG} are to guarantee the conditions in the two steps $(1)$ and $(2)$ mentioned in the preceding subsection. In particular, step (d) shows how each node $i$ defines its quantized interval $\Rcal_i(k)$ and computes $\qbf_i(k)$ and $\qbf_i^{b}(k)$. In addition, it is clear from the definition of $\Rcal_i(k)$ that 
\[
\|\Delta_i(k)\| = \|\xbf_i(k)-\qbf_i(k)\| \lesssim~ \Ocal(\alpha(k)), 
\]
which decays to zero. On the other hand, step (b) shows how node $i$ can recover the quantized value $\qbf_j(k)$ from the encoded bits, $\qbf_j^b(k)$ and $\Rcal_j(k)$. Note that the interval $\Rcal_j(k)$ can be calculated locally at node $i$ at any time $k$ since node $i$ knows the previous value $\qbf_j(k-1)$ for $j\in\Ncal_i$. Thus, Algorithm \ref{alg:QDSG} is fully distributed, that is, the updates at the nodes are executed in parallel and based only on local interactions between nodes. Finally, we show in the next section that by executing Eq.\ \eqref{distributed:QDSG} we have $\xbf_i(k)\in\Rcal_i(k)$ for all $k\geq0$. This observation explains our motivation in defining $\Rcal_i(k)$ in step (d).

%% file: analysis.tex
\section{Main Results}\label{sec:results}
This section establishes rates of convergence for Algorithm~\ref{alg:QDSG} in the cases where $f_i$ are convex and strongly convex.  These results show that under adaptive quantization, the convergence rates of the distributed subgradient algorithm are essentially unaffected by the finite communication bandwidths, except for a constant factor that captures the size of these bandwidths.

The main steps of our analysis are as follows. As we observed in the previous section, the adaptive quantization scheme forces the quantization error $\Delta_i(k)$ to decay to zero at the same rate as the step sizes $\alpha(k)$. 
Our first step is to use this fact to show  that the distance between the estimates $\xbf_i(k)$ to the average $\xbarbf(k)$ converges to zero, implying the nodes eventually reach consensus.
Next we will show that the update of (descent on) $\xbarbf(k)$ mirrors the update of standard centralized subgradient methods.
This allows us to study the convergence rate of Algorithm \ref{alg:QDSG} using the standard outline for the analysis of centralized subgradient methods. 

When the $f_i$ are convex, and we use stepsizes $\alpha(k)= 1/ \,\sqrt{k+1}$, we show that the time-weighted average $\zbf_i(k)$ in \eqref{alg_QDSG:zi} obeys
\[
    f(\zbf_i(k)) - f^* ~\lesssim~
    \frac{1}{(2^b-1)^2(1-\sigma_2)}\cdot \frac{\ln k}{\sqrt{k}},
\]
where $1-\sigma_2$ is the spectral gap that quantifies the connectivity of the underlying network, and $ 1\,/\,(2^b-1)$ is the resolution of the quantizer using $b$ communication bits. When the objective function is strongly convex, using  stepsizes $\alpha(k)= a/ \,(k+1)$ for appropriately chosen constant $a$, we have a refined rate on the convergence of the decision variables themselves, 
\[
    \|\zbf_i(k)-\xbf^*\|^2 ~\lesssim~
    \frac{1}{(2^b-1)^2(1-\sigma_2)}\cdot \frac{\ln k}{k}.
\]
Aside from the constant $ 1\,/\,(2^b-1)$, these results match the standard results for the distributed subgradient method with perfect communication, meaning that the quantization does not qualitatively affect the behavior of the algorithm.

\subsection{Preliminaries}
\label{subsec:prelim}
Given a vector $\vbf\in\Rset^d$ we denote by $\xibf(\vbf)$ the error due to the projection of $\vbf$ on to $\Xcal$, 
\[
    \xibf(\vbf) = \vbf - \left[\vbf\right]_{\Xcal},
\]
and rewrite Eq.\ \eqref{distributed:QDSG} as 
\begin{align}
\begin{aligned}
&\vbf_i(k) = \left(\sum_{j\in\Ncal_i} a_{ij}\xbf_j(k)\right) + \xbf_i(k)-\qbf_i(k) + \sum_{j\in\Ncal_i}a_{ij}(\qbf_j(k)-\xbf_j(k)) - \alpha(k)\gbf_i(\xbf_i(k)), \\
&\xbf_i(k+1) = \left[\vbf_i(k)\right]_{\Xcal} = \vbf_i(k) - \xibf_i(\vbf_i(k)). 
\end{aligned}\label{prelim:xi}
\end{align}
Stacking the $\xbf_i^T$ as rows in the $n\times d$ matrix $\Xbf$ (and doing likewise with the $\vbf_i,\qbf_i,\xibf_i$), we can write the above in matrix form as
\begin{align}
\begin{aligned}
\Vbf(k) &= \Abf\Xbf(k) + (\Ibf-\Abf)(\Xbf(k)-\Qbf(k)) - \alpha(k)\Gbf(\Xbf(k)),\\
\Xbf(k+1) &= \Vbf(k) - \Xibf(\Vbf(k)),
\end{aligned}\label{prelim:X}
\end{align}
where $\Abf$ is the adjacency matrix in Assumption \ref{assump:doub_stoch}.
Let $\xbarbf(k)$ and $\bar{\xibf}(k)$ be the averages of $\xbf_i(k)$ and $\xibf_i(\vbf_i(k))$ across all nodes at time $k$:
\begin{align*}
\xbarbf(k) = \frac{1}{n}\sum_{i=1}^n\xbf_i(k) = \frac{1}{n}\Xbf^T\1 \in\Rset^{d}\quad\text{and}\quad \bar{\xibf}(k) = \frac{1}{n}\sum_{i=1}^n\xibf_i(\vbf_i(k)) =\frac{1}{n} \Xibf(\Vbf(k))^T\1\in\Rset^{d}.   
\end{align*}
Since $\1^T\Abf = \1^T$, \eqref{prelim:X} gives
\begin{align}
\begin{aligned}
\vbarbf(k) &= \xbarbf(k) - \frac{\alpha(k)}{n}\sum_{i=1}^n\gbf_i(\xbf_i(k)),\\
\xbarbf(k+1) &= \vbarbf(k) - \bar{\xibf}(k). 
\end{aligned}\label{prelim:xbar}
\end{align}
Recall that $\Delta_i(k) = \xbf_i(k) - \qbf_i(k)\in\Rset^{d}$ and let $\Delta(k)$ be defined as
\begin{align*}
\Delta(k) = \left[\begin{array}{cc}
-\;\Delta_1(k)^T\;-  \\
\cdots\\
-\;\Delta_n(k)^T\;-
\end{array} \right]\in\Rset^{n\times d}\cdot
\end{align*}

We now consider the following sequence of lemmas, which provides fundamental preliminaries for our main results given in the next sections. In the sequel, we will consider two choices of the step sizes $\alpha(k)$, that is, $\alpha(k) = 1/\sqrt{k+1}$ or $\alpha(k) = 1/(k+1)$. These choices of step sizes also are used to establish our main results in the next section. For ease of exposition we delay all the proofs of the results in this section to Appendix \ref{sec:proofs}.

We first provide an upper bound for the projection error $\xibf_i$ in the following lemma.

\begin{lemma}\label{lem:projection}
Suppose that Assumption \ref{assump:doub_stoch} holds. Let the sequence $\{\xbf_i(k)\}$, for all $i\in\Vcal$, be generated by Algorithm \ref{alg:QDSG}. Then for all $i\in\Vcal$ we have
\begin{align}
\|\,\xibf_i(\vbf_i(k))\,\| \leq \sum_{j\in\Ncal_i}a_{ij}\left\|\Delta_i(k) - \Delta_j(k)\right\| + L_i\alpha(k).\label{lem_project:Ineq1}
\end{align}
In addition, let $L=\sum_{i\in\Vcal}L_i$. Then, we obtain 
\begin{align}
\sum_{i=1}^n\|\,\xibf_i(\vbf_i(k))\,\|^2 \leq 8\sum_{i=1}^n\|\Delta_i(k)\|^2 + 2L^2\alpha^2(k).\label{lem_project:Ineq2}
\end{align}
\end{lemma}

Next we provide an upper bound for the consensus errors $\|\xbf_i(k)-\xbarbf(k)\|$ in the following lemma. 

\begin{lemma}\label{lem:consensus}
Suppose that Assumption \ref{assump:doub_stoch} holds. Let the sequence $\{\xbf_i(k)\}$, for all $i\in\Vcal$, be generated by Algorithm \ref{alg:QDSG}. In addition, let $\{\alpha(k)\}$ be a nonnegative nonincreasing sequence of stepsizes. Then, we have
\begin{align}
\|\Xbf(k+1)-\1\xbarbf(k+1)^T\| &\leq 6\sum_{t=0}^{k}\sigma_2^{k-t}\|\Delta(t)\| + 3L\sum_{t=0}^{k}\sigma_2^{k-t}\alpha(t).\label{lem_consensus:consensus_bound}
\end{align}
\end{lemma}

As mentioned, the main motivation of the adaptive quantization is to eliminate the impact of quantized errors. In particular, we will show that the quantized errors produced by Algorithm \ref{alg:QDSG} decrease to zero at the same rate with the step size $\alpha(k)$. To do that, we require the following technical condition. 
\begin{assump}\label{assump:bits}
Let $\gamma = 48(2+ L)\,/\,(1-\sigma_2)$. Then the number bits $b$ of the communication bandwdith satisfies 
\begin{align}
\sqrt{nd}\gamma \leq 2^b - 1.\label{assumption_bits:Ineq}
\end{align}
\end{assump}
The following lemma is to show that the quantized error $\|\Delta_i(k)\|\lesssim~ \alpha(k)$, for all $i\in\Vcal$, for the case $\sigma_{2}^{k}\leq \alpha(k)$ for all $k\geq0$. When $\sigma_{2}^k\geq \alpha(k)$ for some small $k$, e.g., $\sigma_2$ is closed to $1$ but not equal, one can show from our analysis that $\|\Delta_i(k)\|\lesssim~ \sigma_2^{k} $, which is eventually converge to zero faster than $\alpha(k)$. Since this issue has been captured by the rate of consensus in Eq.\ \eqref{lem_consensus:consensus_bound}, we skip it here for simplicity.

\begin{lemma}\label{lem:quantized_error}
Suppose that Assumption \ref{assump:doub_stoch} and \ref{assump:bits} hold. Let the sequence $\{\xbf_i(k)\}$, for all $i\in\Vcal$, be generated by Algorithm \ref{alg:QDSG}. Let $\alpha(k)$ be either $\alpha(k) = 1\,/\,\sqrt{k+1}$ or $\alpha(k) = 1\,/\,(k+1)$. Then we have all $k\geq 0$ 
\begin{align}
\xbf_i(k+1)\in \Rcal_i(k+1) \triangleq \left[\qbf_i(k) - \frac{\gamma}{2}\alpha(k)\1,\,\qbf_i(k) + \frac{\gamma}{2}\alpha(k)\1\right].\label{lem_quantized_error:Ineq} 
\end{align}
In addition, we also have
\begin{align}
\|\Delta_i(k)\| \leq \frac{\sqrt{d}\gamma}{2^b-1}\alpha(k).\label{lem_quantized_error:error_bound} 
\end{align}
\end{lemma}
The following lemma is a consequence of Lemmas \ref{lem:consensus} and \ref{lem:quantized_error}.

\begin{lemma}\label{lem:consensus_bound}
Suppose that Assumptions \ref{assump:doub_stoch} and \ref{assump:bits} hold. Let the sequence $\{\xbf_i(k)\}$, for all $i\in\Vcal$, be generated by Algorithm \ref{alg:QDSG}. Let $\alpha(k)$ be either $\alpha(k) = 1\,/\,\sqrt{k+1}$ or $\alpha(k) = 1\,/\,(k+1)$. Then, we have
\begin{align}
\lim_{k\rightarrow\infty} \xbf_i(k) = \lim_{k\rightarrow\infty} \xbf_j(k),\qquad \forall\, i,j\in\Vcal.\label{lem_consensus:asymp_conv}
\end{align}
In addition, if $\alpha(k)$ is also square-summable, i.e., 
\begin{align}
\sum_{k=0}^{\infty}\alpha^2(k) < \infty,\label{lem_consensus:square_summable}
\end{align}
then for all $k\geq0$ we have 
\begin{align}
\sum_{t=0}^{k}\alpha(t) \|\Xbf(t)-\1\xbarbf(t)^T\| \leq \left(\frac{6\sqrt{nd}\gamma+3L2^b}{(1-\sigma_2)(2^b-1)}\right)\sum_{t=0}^{k}\alpha^2(t)<\infty.\label{lem_consensus:finite_sum} 
\end{align}
\item If $\alpha(k) = 1\,/\,\sqrt{k+1}$ then we have for all $k\geq0,$
\begin{align}
\sum_{t=0}^{k}\alpha(t)\|\Xbf(t)-\1\xbarbf(t)^T\|\leq \left(\frac{6\sqrt{nd}\gamma+3L2^b}{(1-\sigma_2)(2^b-1)}\right)(\ln(k+1)+1). \label{lem_consensus:rate}
\end{align}
\end{lemma} 
Finally, we provide an upper bound for the optimal distance $\|\xbarbf(k)-\xbf^*\|^2$ in the following lemma. 

\begin{lemma}\label{lem:opt_dist}
Suppose that Assumptions \ref{assump:doub_stoch} and \ref{assump:bits} hold. Let the sequence $\{\xbf_i(k)\}$, for all $i\in\Vcal$, be generated by Algorithm \ref{alg:QDSG}. In addition, let $\xbf^*\in\Xcal^*$ be a solution of problem \eqref{prob:obj}. Let $\alpha(k)$ be either $\alpha(k) = 1\,/\,\sqrt{k+1}$ or $\alpha(k) = 1\,/\,(k+1)$. Then, we have
\begin{align}
\|\xbarbf(k+1)-\xbf^*\|^2&\leq \|\xbarbf(k)-\xbf^*\|^2-\frac{2\alpha(k)}{n}\sum_{i=1}^n\gbf_i(\xbf_i(k))^{T}(\xbf_i(k)-\xbf^*) \notag\\
&\qquad +\frac{2\left(4\sqrt{nd}\gamma+3L2^{b}\right)}{\sqrt{n}(2^b-1)}\alpha(k)\|\Xbf(k)-\1\xbarbf(k)^T\|\notag\\ 
&\qquad + \frac{2(4\sqrt{nd}\gamma + 3L2^{b})^2}{(2^b-1)^2\sqrt{n}}\alpha^2(k)\cdot\label{lem_opt_dist:Ineq}
\end{align}
\end{lemma}

\subsection{Convergence Results of Convex Functions}
We now present the first main result of this paper, which is the rate of convergence of Algorithm \ref{alg:QDSG} to the optimal value of problem \eqref{prob:obj} when the local functions $f_i$ are convex. Since the update of $\xbarbf(k)$ in Eq.\ \eqref{prelim:xbar} can be viewed as a variant of a centralized projected subgradient methods used to solve problem \eqref{prob:obj}, we utilize  standard techniques in the analysis of these methods to derive the rate of convergence of Algorithm \ref{alg:QDSG}. Specifically, at any time $k \geq 0$ if each node $i \in \Vcal$ maintains a variable $\zbf_i(k)$ to compute
the time-weighted average of its estimate $\xbf_i(k)$ and if the stepsize $\alpha(k)$ decays as $\alpha(k) = 1/ \,\sqrt{k+1}$, the objective function value $f$ in Eq.\ \eqref{prob:obj} estimated at each $\zbf_i(k)$ converges to the optimal value with a rate $\mathcal{O}\left(\eta \ln(k+1)\, /\,\sqrt{k+1}\right)$, where $\eta$ is some constant depending on the algebraic connectivity $1-\sigma_2$ of the network, the number of quantized bits $b$, and the Lipschitz constants $L_i$ of $f_i$. We also note that this condition on the stepsizes is also used to study the convergence rate of centralized subgradient methods \cite{Nesterov2004}. The following theorem is used to show the convergence rate of Algorithm \ref{alg:QDSG}. 
 
\begin{theorem}\label{thm_convex:rate}
Suppose that Assumptions \ref{assump:doub_stoch} and \ref{assump:bits} hold. Let the sequence $\{\xbf_i(k)\}$, for all $i\in\Vcal$, be generated by Algorithm \ref{alg:QDSG}. In addition, let $\alpha(k) = 1\,/\,\sqrt{k+1}$. Moreover, suppose that each node $i$, for all $i\in\Vcal$, stores a variable $\zbf_i\in\Rset^{d}$ initiated arbitrarily in $\Xcal$ and updated as 
\begin{align}
\zbf_i(k) = \frac{\sum_{t=0}^k\alpha(t)\xbf_i(t)}{\sum_{t=0}^k\alpha(t)},\quad\forall i\in\Vcal.\label{thm_convex_rate:zi}
\end{align} 
Then for all $i\in\Vcal$ and $k\geq 0$ we have
\begin{align}
f(\zbf_{i}(k)) - f^*&\leq\frac{n \|\xbarbf(0)-\xbf^*\|^2}{2\sqrt{k+1}}+ \frac{\sqrt{n}(6\sqrt{nd}\gamma + 5L2^{b})^2}{(1-\sigma_2)(2^b-1)^2}\frac{(\ln(k+1)+1)}{\sqrt{k+1}}\cdot\label{thm_convex_rate:Ineq}
\end{align}
\end{theorem} 
\begin{proof}
For convenience, let $\rbf(k) = \xbarbf(k)-\xbf^*$, where $\xbf^*\in\Xcal^*$ is a solution of problem \eqref{prob:obj}. By Eq.\ \eqref{lem_opt_dist:Ineq} we have
\begin{align*}
\|\rbf(k+1)\|^2 &\leq \|\rbf(k)\|^2-\frac{2\alpha(k)}{n}\sum_{i=1}^n\gbf_i(\xbf_i(k))^{T}(\xbf_i(k)-\xbf^*) + \frac{2(4\sqrt{nd}\gamma + 3L2^{b})^2}{(2^b-1)^2\sqrt{n}}\alpha^2(k) \notag\\
&\qquad +\frac{2\left(4\sqrt{nd}\gamma+3L2^{b}\right)}{\sqrt{n}(2^b-1)}\alpha(k)\|\Xbf(k)-\1\xbarbf(k)^T\|,
\end{align*}
which by the convexity of $f_i$ yields
\begin{align}
\|\rbf(k+1)\|^2&\leq \|\rbf(k)\|^2 + \frac{2(4\sqrt{nd}\gamma + 3L2^{b})^2}{(2^b-1)^2\sqrt{n}}\alpha^2(k) \notag\\
&\qquad +\frac{2\left(4\sqrt{nd}\gamma+3L2^{b}\right)}{\sqrt{n}(2^b-1)}\alpha(k)\|\Xbf(k)-\1\xbarbf(k)^T\|\notag\\ 
&\qquad - \frac{2\alpha(k)}{n}\left(\sum_{i=1}f_i(\xbf_i(k))-f_i(\xbf^*))\right)
,\label{thm_convex_rate:Eq1}
\end{align}
We now analyze the last term on the right-hand side of Eq.\ \eqref{thm_convex_rate:Eq1}. Indeed, by Eq.\ \eqref{prop:bounded_subg_ineq} and using $f = \sum_{i=1}^n f_i$ and $f^* = f(x^*)$, we have for a fixed $\ell\in\Vcal$
\begin{align}
&- \sum_{i=1}^n f_i(\xbf_i(k)) - f_i(\xbf^*) = - \sum_{i=1}^n\Big( f_i(\xbf_i(k)) -f_i(\xbarbf(k)) + f_i(\xbarbf(k)) - f_i(\xbf^*)\Big)\notag\\
&\qquad\leq \sum_{i=1}^n L_i |\,\xbf_i(k)-\xbarbf(k)\,| - \Big(f(\xbarbf(k)) - f^*\Big)\notag\\
&\qquad \leq L \|\,\Xbf(k)-\1\xbarbf(k)^T\,\| - \Big(f(\xbarbf(k)) -f(\xbf_{\ell}(k)) +f(\xbf_{\ell}(k)) - f^*\Big)\notag\\
&\qquad \leq 2L\|\,\Xbf(k)-\1\xbarbf(k)^T\,\| - \Big(f(\xbf_{\ell}(k)) - f^*\Big),
\label{thm_convex_rate:Eq1a}
\end{align}
which when substituting into Eq.\ \eqref{thm_convex_rate:Eq1} yields
\begin{align}
\|\rbf(k+1)\|^2 &\leq \|\rbf(k)\|^2 + \frac{2(4\sqrt{nd}\gamma + 3L2^{b})^2}{(2^b-1)^2\sqrt{n}}\alpha^2(k) - \frac{2}{n}\alpha(k)\Big(f(\xbf_{\ell}(k)) - f^*\Big) \notag\\
&\qquad +\left(\frac{2\left(4\sqrt{nd}\gamma+3L2^{b}\right)}{\sqrt{n}(2^b-1)})+ \frac{4L}{n}\right)\alpha(k)\|\Xbf(k)-\1\xbarbf(k)^T\|\notag\\ 
&= \|\rbf(k)\|^2 + \frac{2(4\sqrt{nd}\gamma + 3L2^{b})^2}{(2^b-1)^2\sqrt{n}}\alpha^2(k)  - \frac{2}{n}\alpha(k)\Big(f(\xbf_{\ell}(k)) - f^*\Big)\notag\\
&\qquad +\frac{2\left(4\sqrt{nd}\gamma+5L2^{b}\right)}{\sqrt{n}(2^b-1)}\alpha(k)\|\Xbf(k)-\1\xbarbf(k)^T\|\label{thm_convex_rate:Eq1b},
\end{align}
which when iteratively updating over $k = 0,\ldots,K$ for some $K\geq0$ we have
\begin{align*}
\|\rbf(K+1)\|^2 &\leq \|\rbf(0)\|^2 + \frac{2(4\sqrt{nd}\gamma + 3L2^{b})^2}{(2^b-1)^2\sqrt{n}}\sum_{k=0}^{K}\alpha^2(k) \notag\\
&\qquad +\frac{2\left(4\sqrt{nd}\gamma+5L2^{b}\right)}{\sqrt{n}(2^b-1)}\sum_{k=0}^{K}\alpha(k)\|\Xbf(k)-\1\xbarbf(k)^T\|\notag\\ 
&\qquad - \frac{2}{n}\sum_{k=0}^{K}\alpha(k)\Big(f(\xbf_{\ell}(k)) - f^*\Big).
\end{align*}
Since $\alpha(k) = 1/\sqrt{k+1}$ we now use Eq.\ \eqref{lem_consensus:rate} into the preceding relation to have
\begin{align*}
\|\rbf(K+1)\|^2 & \leq \|\rbf(0)\|^2 + \frac{2(4\sqrt{nd}\gamma + 3L2^{b})^2}{(2^b-1)^2\sqrt{n}}(\ln(K+1)+1) \notag\\
&\qquad +\frac{2\left(4\sqrt{nd}\gamma+5L2^{b}\right)}{\sqrt{n}(2^b-1)}\left(\frac{6\sqrt{nd}\gamma+3L2^b}{(1-\sigma_2)(2^b-1)}\right)(\ln(K+1)+1)\notag\\
&\qquad - \frac{2}{n}\sum_{k=0}^{K}\alpha(k)\Big(f(\xbf_{\ell}(k)) - f^*\Big)\notag\\
&\leq \|\rbf(0)\|^2 + \frac{4(6\sqrt{nd}\gamma + 5L2^{b})^2}{(1-\sigma_2)(2^b-1)^2\sqrt{n}}(\ln(K+1)+1) \notag\\
&\qquad - \frac{2}{n}\sum_{k=0}^{K}\alpha(k)\Big(f(\xbf_{\ell}(k)) - f^*\Big).
\end{align*}
Rearranging the preceding relation and dropping the nonnegative $\|\rbf(K+1)\|$ we obtain
\begin{align*}
\sum_{k=0}^{K}\alpha(k)\Big(f(\xbf_{\ell}(k)) - f^*\Big)&\leq \frac{n \|\rbf(0)\|^2}{2}+ \frac{2\sqrt{n}(6\sqrt{nd}\gamma + 5L2^{b})^2}{(1-\sigma_2)(2^b-1)^2}(\ln(K+1)+1),
\end{align*}
which by dividing both sides by $\sum_{k=0}^{K}\alpha(k)$ and using the convexity of $f$ gives Eq.\  \eqref{thm_convex_rate:Ineq}, i.e.,
\begin{align*}
f(\zbf_{\ell}(K)) - f^*&\leq \frac{n \|\rbf(0)\|^2}{2\sqrt{K+1}}+ \frac{2\sqrt{n}(6\sqrt{nd}\gamma + 5L2^{b})^2}{(1-\sigma_2)(2^b-1)^2}\frac{(\ln(K+1)+1)}{\sqrt{K+1}},
\end{align*}
where in the last inequality we use the integral test for $K\geq 0$ to have
\begin{align*}
\sum_{k=0}^{K}\alpha(k) = \sum_{k=0}^{K}\frac{1}{\sqrt{k+1}}\geq \int_{t=0}^{K+1}\frac{1}{\sqrt{t+1}}dt = 2(\sqrt{K+2}-1)\geq \sqrt{K+1}.
\end{align*}
\end{proof}
It is worth to mention that under the choice of $\alpha(k) = 1\,/\,(k+1)$, for all $k\geq 0$, one can show that $\xbf_i(k)$ asymptotically converges to $\xbf^*$ for all $i\in\Vcal$. This is a consequence of Lemmas \ref{lem:consensus} and \ref{lem:opt_dist}, and some standard analysis. The following lemma states this result. The analysis is omitted and can be found in \cite[Theorem 3]{DoanMR2018a}.  
\begin{lemma}
Suppose that Assumptions \ref{assump:doub_stoch} and \ref{assump:bits} hold. Let the sequence $\{\xbf_i(k)\}$, for all $i\in\Vcal$, be generated by Algorithm \ref{alg:QDSG}. Let $\alpha(k) = 1\,/\,(k+1)$.  Then we obtain
\begin{align}
\lim_{k\rightarrow\infty} \xbf_i(k) = \xbf^*,\qquad \text{for all } i\in\Vcal,
\end{align}
for some  $\xbf^*$ that is a solution of problem \eqref{prob:obj}.
\end{lemma}

\subsection{Convergence Results of Strongly Convex Case}
In this section, our goal is to study the convergence rate of Algorithm \ref{alg:QDSG} when the local functions $f_i$ are strongly convex, that is, we make the following assumption on $f_i$ 
\begin{assump}\label{assump:sconvexity}
Each function $f_i$ is strongly convex with some positive constant $\mu_i$, i.e., the condition \eqref{notation:sc} holds. 
\end{assump}
Under this assumption, we show that if each node $i \in \Vcal$ maintains a variable $\zbf_i(k)$ to compute the time average of its estimate $\xbf_i(k)$ and if the stepsize $\alpha(k)$ decays as $\alpha(k) = a\,/\, (k+1)$ for some properly chosen constant $a$, the variable $\zbf_i(k)$ converges to the optimal solution $x^*$ of problem \eqref{prob:obj} with a rate $\mathcal{O}\left(\eta \ln(k+1)\, /\,(k+1)\right)$, where $\eta$ is some constant depending on the algebraic connectivity $1-\sigma_2$ of the network, the number of quantized bits $b$, and the constants $L_i$ and $\mu_i$ of $f_i$. The following theorem is used to show the convergence rate of Algorithm 1 under Assumption \ref{assump:sconvexity}. 
\begin{theorem}\label{thm_sconvex:rate}
Suppose that Assumptions \ref{assump:doub_stoch} and \ref{assump:sconvexity} hold. Let the sequence $\{\xbf_i(k)\}$, for all $i\in\Vcal$, be generated by Algorithm \ref{alg:QDSG}. We denote by $\mu=\min_{i\in\Vcal}\mu_i$. In addition, let $\{\alpha(k)\} = a\,/\,k+1$ for some $a\geq 1\,/\,\mu$. Moreover, suppose that each node $i$, for all $i\in\Vcal$, stores a variable $\zbf_i\in\Rset$ initiated arbitrarily in $\Xcal$ and updated as 
\begin{align}
\zbf_i(k) = \frac{\sum_{t=0}^k \xbf_i(t)}{k+1},\quad\forall i\in\Vcal.\label{thm_sconvex_rate:zi}
\end{align} 
Let $\xbf^*\in\Xcal^*$ be a solution of problem \eqref{prob:obj}. Then for all $i\in\Vcal$ and $k\geq 0$ we have
\begin{align}
\|\zbf_{i}(k)-\xbf^*\|^2 \leq \frac{4\sqrt{n}\alpha(0)(6\sqrt{nd}\gamma + 5L2^{b})^2}{(1-\sigma_2)(2^b-1)^2}\frac{1+\ln(k+1)}{k+1}\cdot\label{thm_sconvex_rate:Ineq}
\end{align}
\end{theorem}

\begin{proof}
Let $\xbf^*$ be a solution of problem \eqref{prob:obj}. For convenience, let $\rbf(k) = \xbarbf(k)-\xbf^*.$ By Eq.\ \eqref{lem_opt_dist:Ineq} we have
\begin{align}
\|\rbf(k+1)\|^2 &\leq \|\rbf(k)\|^2-\frac{2\alpha(k)}{n}\sum_{i=1}^n\gbf_i(\xbf_i(k))^{T}(\xbf_i(k)-\xbf^*) \notag\\
&\qquad +\frac{2\left(4\sqrt{nd}\gamma+3L2^{b}\right)}{\sqrt{n}(2^b-1)}\alpha(k)\|\Xbf(k)-\1\xbarbf(k)^T\|\notag\\ 
&\qquad + \frac{2(4\sqrt{nd}\gamma + 3L2^{b})^2}{(2^b-1)^2\sqrt{n}}\alpha^2(k)\notag\\
&\leq \|\rbf(k)\| + \frac{2(4\sqrt{nd}\gamma + 3L2^{b})^2}{(2^b-1)^2\sqrt{n}}\alpha^2(k)\notag\\
&\qquad +\frac{2\left(4\sqrt{nd}\gamma+3L2^{b}\right)}{\sqrt{n}(2^b-1)}\alpha(k)\|\Xbf(k)-\1\xbarbf(k)^T\|\notag\\ 
&\qquad -\frac{2\alpha(k)}{n}\sum_{i=1}^n \Big(f_i(\xbf_i(k)) - f_i(\xbf^*)+\frac{\mu_i}{2}\|\xbf_i(k)-\xbf^*\|^2\Big),\label{thm_sconvex_rate:Eq1}
\end{align}
where the last inequality is due to the strong convexity of $f_i$, i.e., Eq.\ \eqref{notation:sc}. First, using the Jensen's inequality on quadratic function $(\cdot)^2$ we have
\begin{align*}
-\frac{1}{n}\sum_{i=1}^n\mu_i\|\xbf_i(k)-\xbf^*\|^2\leq -\mu\,\frac{1}{n}\sum_{i=1}^n\|\xbf_i(k)-\xbf^*\|^2\leq -\mu\|\xbarbf(k)-\xbf^*\|^2 = -\mu\,\|\rbf(k)\|^2.
\end{align*}
Fix some $\ell\in\Vcal$. Then, substituting the preceding relation into Eq.\ \eqref{thm_sconvex_rate:Eq1} and using Eq.\ \eqref{thm_convex_rate:Eq1a} yield
\begin{align*}
\|\rbf(k+1)\|^2 
&\leq \left(1-\mu\alpha(k)\right)\|\rbf(k)\|^2  + \frac{2(4\sqrt{nd}\gamma + 3L2^{b})^2}{(2^b-1)^2\sqrt{n}}\alpha^2(k)\notag\\
&\qquad +\frac{2\left(4\sqrt{nd}\gamma+3L2^{b}\right)}{\sqrt{n}(2^b-1)}\alpha(k)\|\Xbf(k)-\1\xbarbf(k)^T\|\notag\\ 
&\qquad -\frac{2\alpha(k)}{n}\sum_{i=1}^n \Big(f_i(\xbf_i(k)) - f_i(\xbf^*)\Big)\notag\\
&\stackrel{\eqref{thm_convex_rate:Eq1a} }{\leq} \left(1-\mu\alpha(k)\right)\|\rbf(k)\|^2  + \frac{2(4\sqrt{nd}\gamma + 3L2^{b})^2}{(2^b-1)^2\sqrt{n}}\alpha^2(k)\notag\\
&\qquad +\frac{2\left(4\sqrt{nd}\gamma+5L2^{b}\right)}{\sqrt{n}(2^b-1)}\alpha(k)\|\Xbf(k)-\1\xbarbf(k)^T\|\notag\\ 
&\qquad  - \frac{2}{n}\alpha(k)\Big(f(\xbf_{\ell}(k)) - f^*\Big),
% \notag\\
% &\leq  \left(1-\mu\alpha(k)\right)\|\rbf(k)\|^2 + \frac{4(6\sqrt{nd}\gamma + 5L2^{b})^2}{\sqrt{n}(1-\sigma_2)(2^b-1)^2}\alpha^2(k) \notag\\
% &\qquad - \frac{2}{n}\alpha(k)\Big(f(\xbf_{\ell}(k)) - f^*\Big),
\end{align*}
%where in the last inequality we use Eq.\ \eqref{lem_consensus:finite_sum}. 
Note that $\alpha(k) = a\,/\,(k+1)$ with $a\geq 1\,/\,\mu$, implying $\mu\alpha(k)\geq 1\,/\,(k+1)$. Thus, the preceding equation gives 
\begin{align*}
\|\rbf(k+1)\|^2 &\leq \frac{k}{k+1}\|\rbf(k)\|^2 + \frac{2(4\sqrt{nd}\gamma + 3L2^{b})^2}{(2^b-1)^2\sqrt{n}}\alpha^2(k)\notag\\
&\qquad +\frac{2\left(4\sqrt{nd}\gamma+5L2^{b}\right)}{\sqrt{n}(2^b-1)}\alpha(k)\|\Xbf(k)-\1\xbarbf(k)^T\|\notag\\ 
&\qquad  - \frac{2}{n}\alpha(k)\Big(f(\xbf_{\ell}(k)) - f^*\Big).
\end{align*}
Multiplying both sides of the preceding equation by $k+1$, and using $(k+1)\,/\,k\leq 2$ and $\alpha(k) = \alpha(0) / (k+1)$ we have
\begin{align}
(k+1)\|\rbf(k+1)\|^2 &\leq k\|\rbf(k)\|^2 +  \frac{4\alpha(0)(4\sqrt{nd}\gamma + 3L2^{b})^2}{(2^b-1)^2\sqrt{n}}\alpha(k)\notag\\
&\qquad +\frac{2\left(4\sqrt{nd}\gamma+5L2^{b}\right)}{\sqrt{n}(2^b-1)}\|\Xbf(k)-\1\xbarbf(k)^T\|\notag\\ 
&\qquad - \frac{2\alpha(0)}{n}\Big(f(\xbf_{\ell}(k)) - f^*\Big),\label{thm_sconvex_rate:Eq1a}
\end{align}
By Eq.\ \eqref{lem_consensus:consensus_bound} and using $\|\Delta(t)\|\leq \alpha(t)$ we have
\begin{align}
\sum_{k=0}^{K}\|\Xbf(k)-\1\xbarbf(k)^T\| &\leq (3L+6)\sum_{k=0}^{K}\sum_{t=0}^{k-1}\sigma_2^{k-t}\alpha(t)\notag\\
&\leq (3L+6)\sum_{k=0}^{K-1}\alpha(k)\sum_{t=k+1}^{K}\sigma_2^{t}\leq \frac{3L+6}{1-\sigma_2}\sum_{k=0}^{K-1}\alpha(k).
\end{align}
Next, summing up both sides of Eq. \eqref{thm_sconvex_rate:Eq1a} over $k = 0,\ldots,K$ for some $K\geq 0$, using the preceding relation, and rearranging we obtain  
\begin{align*}
\frac{2\alpha(0)}{n}\sum_{k=0}^{K}\Big(f(x_{\ell}(k)) - f^*\Big) &\leq \frac{2\alpha(0)(6\sqrt{nd}\gamma + 5L2^{b})^2}{\sqrt{n}(1-\sigma_2)(2^b-1)^2}\sum_{k=0}^{K}\alpha(k)\notag\\
&\leq \frac{2\alpha(0)(6\sqrt{nd}\gamma + 5L2^{b})^2}{\sqrt{n}(1-\sigma_2)(2^b-1)^2}(\ln(K+1)+1),
\end{align*}
which when dividing both sides by $(K+1)\,/\,n$ and using the convexity of $f$ yields 
\begin{align*}
&2\alpha(0)\Big[f(\zbf_{\ell}(K)) - f^*\Big] \leq \frac{2\sqrt{n}\alpha(0)(6\sqrt{nd}\gamma + 5L2^{b})^2}{(1-\sigma_2)(2^b-1)^2}\frac{1+\ln(K+1)}{K+1}.
\end{align*}
Since the functions $f_i$ are strongly convex with constant $\mu_i$, $f$ is strongly convex with constant $\mu$. Thus, using the preceding equation and $\alpha(0) = a\geq 1\,/\,\mu$ gives Eq.\ \eqref{thm_sconvex_rate:Ineq}, i.e.,
\begin{align*}
&\|\zbf_{\ell}(K)-\xbf^*\|^2 \leq \frac{2}{\mu}\left[f(\zbf_{\ell}(K)) - f^* \right]\leq \frac{2\sqrt{n}\alpha(0)(6\sqrt{nd}\gamma + 5L2^{b})^2}{(1-\sigma_2)(2^b-1)^2}\frac{1+\ln(K+1)}{K+1}\cdot
\end{align*}
\end{proof}

%% file: simulations.tex
%!TEX root = adaptive_quantization.tex

\section{Simulations}\label{sec:simulation}
In this section, we apply Algorithm \ref{alg:QDSG} for solving linear regression problems, the most popular technique for data fitting \cite{HTF2009,SSBD2014} in statistical machine learning, over a network of processors under random quantization. The goal of this problem is to find a linear relationship between a set of variables and some real value outcome. That is, given a training set $S = \{(\abf_i,b_i)\in\mathbb{R}^{d}\times\mathbb{R}\}$ for $i=1,\ldots,n$, we want to learn a parameter $\xbf$ that minimizes 
\begin{align*}
\min_{\xbf\in\Xcal}\sum_{i=1}^n f_i(\xbf;\abf_i,b_i),
\end{align*}
where $\Xcal = [-1\,,\, 1]^{d}$ and $d=10$, i.e., $\xbf,\, \abf_i\in\Rset^{10}$. Here, $f_i$ are the loss functions defined over the dataset. For the purpose of our simulation, we will consider two loss functions, namely, quadratic loss and absolute loss functions. While the quadratic loss is strongly convex, the absolute loss is only convex.         

First, when $f_i$ are quadratic, we have the well-known least square problem
\begin{align*}
\min_{\xbf\in\Xcal}\;\sum_{i=1}^n(\abf_i^T\xbf-b_i)^2.
\end{align*}
Second, regression problems with absolute loss functions (or L$1$ norm) is often referred to as robust regression, which is known to be robust to outliers \cite{karst1958linear}, given as follows
\begin{align*}
\min_{\xbf\in\Xcal}\;\sum_{i=1}^n |\, \abf_i^T\xbf-b_i\,|.
\end{align*} 

We consider simulated training data sets, i.e., $(\abf_i,b_i)$ are generated randomly with uniform distribution between $[0,1]$.  We consider the performance of the distributed subgradient methods on an undirected connected graph of $50$ nodes, i.e., $\Gcal=(\Vcal,\Ecal)$ and $n = |\Vcal|=100$. Our graph is generated as follows. 
\begin{enumerate}
\item In each network, we first randomly generate the nodes' coordinates in the plane with uniform distribution.
\item Then any two nodes are connected if their distance is less than a reference number $r$, e.g, $r = 0.4$ for our simulations.
\item Finally we check whether the network is connected. If not we return to step $1$ and run the program again.
\end{enumerate}
To implement our algorithm, the adjacency matrix $\Abf$ is chosen as a lazy Metropolis matrix corresponding to $\Gcal$, i.e.,
\begin{align}
\Abf = [a_{ij}] = \left\{\begin{array}{ll}
\frac{1}{2(\max\{|\mathcal{N}_i| , |\Ncal_j|\})}, & \text{ if } (i,j) \in \Ecal\\
0, &\text{ if } (i,j)\notin\Ecal \text{ and } i\neq j\\
1-\sum_{j\in\Ncal_i}a_{ij},& \text{ if } i = j
\end{array}\right.\nonumber
\end{align}
It is obvious to see that the lazy Metropolis matrix $\Abf$ satisfies Assumption \ref{assump:doub_stoch}.

\subsection{Convergence of Function Values}

  \begin{figure} 
   \begin{subfigure}[b]{0.5\textwidth}
         \centering
     \includegraphics[width=\textwidth]{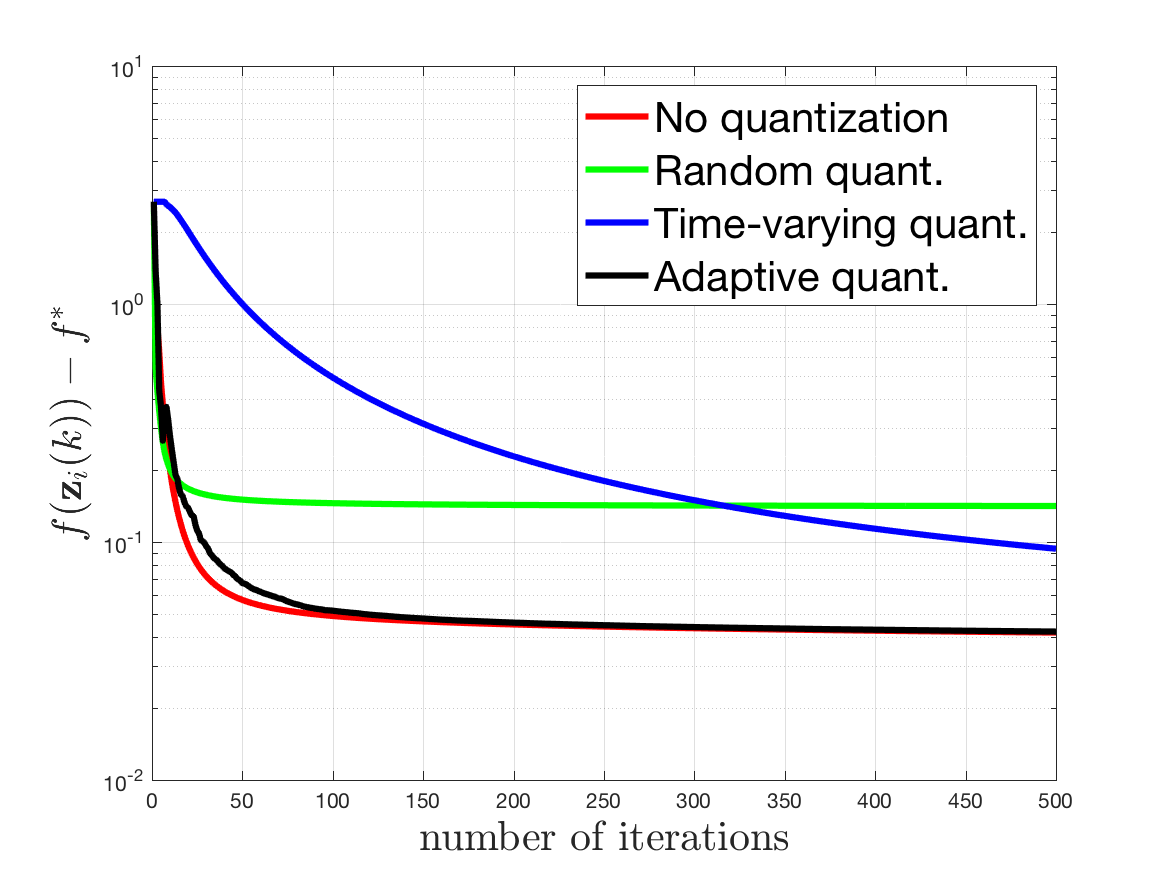}
     \caption{Quadratic loss functions}
      \label{fig:sc_func_value}
     \end{subfigure} 
   \begin{subfigure}[b]{0.5\textwidth}
         \centering \vspace{0.2cm}
     \includegraphics[width=\textwidth]{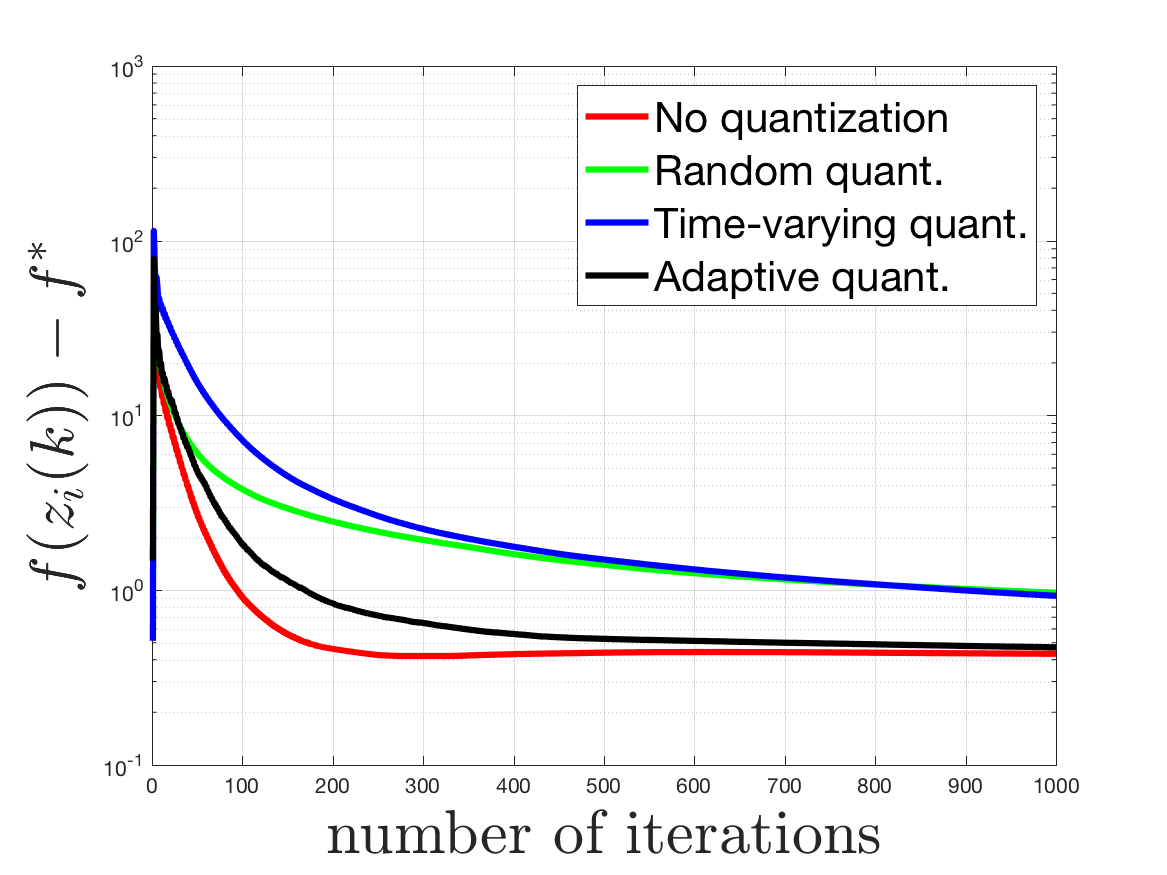}
         \caption{Absolute loss functions}
          \label{fig:c_func_value}
     \end{subfigure} 
  \caption{The convergence of function values using distributed subgradient methods without (\protect\redline) \cite{NedicOP2010}, with random  (\protect\greenline) \cite{DoanMR2018b}, with time-varying  (\protect\blueline) \cite{DoanMR2018a}, and with adaptive  (\protect\blackline) quantization for $n=100$ and $d=10$ are illustrated.}
  \label{fig:func_value}
  \end{figure}
In this simulation, we apply variants of distributed subgradient methods for solving the linear regression problems. In particular, we compare the performance of such methods for three different scenarios, namely, {\sf DSG} with no quantization (i.e., Eq.\ \eqref{distributed:DSG}), {\sf DSG} with time-varying quantization in \cite{DoanMR2018a}, distributed stochastic approximation under random quantization \cite{DoanMR2018b}, and the proposed Algorithm \ref{alg:QDSG} with adaptive quantization. In addition, we use $8$ bits as the size of the nodes' communication bandwidths. The plots in Fig.\ \ref{fig:func_value} show the convergence of these four methods for both quadratic and absolute loss functions.  
 
Note that, DSG with time-varying quantization  \cite{DoanMR2018a} achieves the same rate of convergence as the one with no quantization \cite{NedicOP2010}, but requires that the nodes eventually exchange an infinite number of bits. On the other hand, DSG with random quantization  \cite{DoanMR2018b} only requires a finite number of bits, but achieves a slow rate of convergence. The adaptive quantization in this paper achieves both benefits of time-varying quantization \cite{DoanMR2018a} and random quantization \cite{DoanMR2018b}, i.e., it achieves the same rate as the algorithm without quantization but only using a finite number of bits. In addition, as observed in Fig.\ref{fig:sc_func_value} for quadratic loss and in Fig. \ref{fig:c_func_value} for absolute loss, Algorithm \ref{alg:QDSG} performs almost as well as the one without quantization \cite{NedicOP2010}, and significantly better than the algorithms in \cite{DoanMR2018a,DoanMR2018b}

 \begin{figure} 
   \begin{subfigure}[b]{0.5\textwidth}
        \centering
    \includegraphics[width=\textwidth]{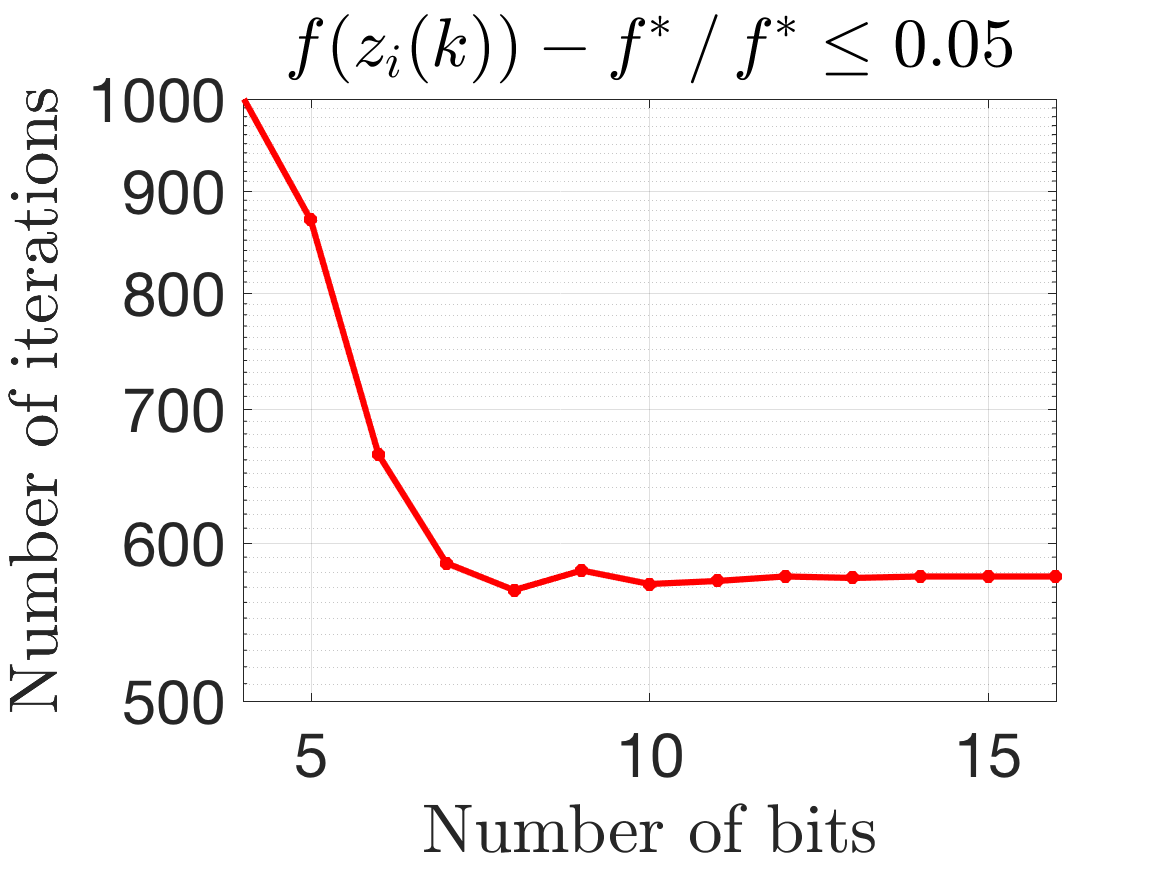}
    \caption{Quadratic loss functions}
    \end{subfigure} 
   \begin{subfigure}[b]{0.5\textwidth}
        \centering    
    \includegraphics[ width = \textwidth]{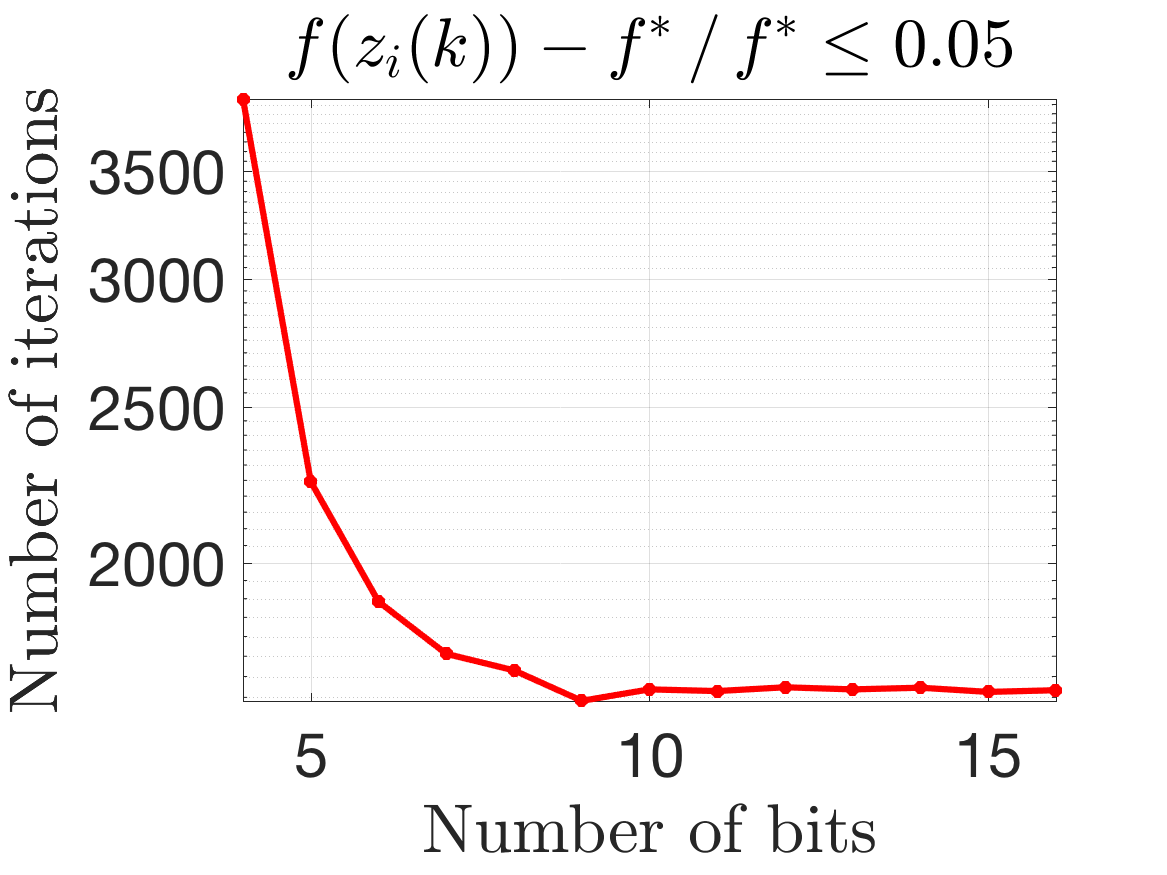}
        \caption{Absolute loss functions}
    \end{subfigure} 
 \caption{The number of iterations as a function of $b$ using distributed subgradient methods with adaptive  quantization for $n=100$ and $d=10$ are illustrated.}
 \label{fig:b_impact}
 \end{figure}

\subsection{Impact of the Number of Bits, $b$}
We now consider the impact of the number of bits $b$ on the performance of Algorithm \ref{alg:QDSG}. In Fig. \ref{fig:b_impact} we plots the number of iterations, needed to obtain the relative error $f(z_i(k))-f^*\,/\,f^*\leq 0.05$, as a function of $b$. We see that the more bits we use, the faster the algorithm converges. Moreover, even when only a very small number of bits, for example, b= 4 are used, the algorithm still works very well. Finally, these plots appear to describe the  curve $1/(2^b-1)^2$ upto some constant in the upper bound of convergence rates given in Theorems \ref{thm_convex:rate} and \ref{thm_sconvex:rate}. This implies that the simulation seems to agree with our results.

%% file: results_proofs.tex
%!TEX root = adaptive_quantization.tex
\newpage
\appendix
\section{Proofs of Results in Section \ref{subsec:prelim} }\label{sec:proofs}

\subsection{Proof of Lemma \ref{lem:projection}}
\begin{proof}
For convenience, we use  $\wbf_i(k)$ to denote $\sum_{j\in\Ncal_i}a_{ij}\xbf_j(k)$ only in this proof. Since $\xbf_i(k)\in\Xcal$, by convexity of $\Xcal$, we have that $\wbf_i(k)\in\Xcal$.
%\begin{align*}
%\wbf_i(k) = \sum_{j\in\Ncal_i}a_{ij}\xbf_j(k)\in\Xcal,    
%\end{align*}
%where we note that $\xbf_i(k)\in\Xcal$. 
Recall that $\Delta_i(k) = \xbf_i(k) - \qbf_i(k)$. By the definition of the projection and using Eqs.\ \eqref{prop:bounded_subg_ineq} and \eqref{prelim:xi} we obtain Eq.\ \eqref{lem_project:Ineq1}, i.e., 
\begin{align*}
\|\xibf_i(\vbf_i(k))\| &=  \left\| \vbf_i(k)-\left[\vbf_i(k)\right]_{\Xcal}\right\|^2\leq \left\|\vbf_i(k) - \wbf_i(k)\right\|^2\notag\\
&= \left\|\Delta_i(k) - \sum_{j\in\Ncal_i}a_{ij}\Delta_j(k) -\alpha(k)\gbf_i(\xbf_i(k))\right\|^2\notag\\
&\leq \left\|\Delta_i(k) - \sum_{j\in\Ncal_i}a_{ij}\Delta_j(k)\right\| + L_i\alpha(k)\notag\\
&\leq  \sum_{j\in\Ncal_i}a_{ij}\left\|\Delta_i(k) -\Delta_j(k)\right\| + L_i\alpha(k).
\end{align*}
Using the preceding inbequality also yields Eq.\ \eqref{lem_project:Ineq2}, i.e., 
\begin{align*}
\sum_{i=1}^n\|\xibf_i(\vbf_i(k))\|^2 &\leq  2\sum_{i=1}^n\left\|\Delta_i(k) - \sum_{j\in\Ncal_i}a_{ij}\Delta_j(k)\right\|^2 + 2\sum_{i=1}^n L_i^2\alpha^2(k)\notag\\
&\leq 2\sum_{i=1}^n\sum_{j\in\Ncal_i}a_{ij}\left\|\Delta_i(k) - \Delta_j(k)\right\|^2 + 2L^2\alpha^2(k)\notag\\
&\leq 8\sum_{i=1}^n\|\Delta_i(k)\|^2 + 2L^2\alpha^2(k),  
\end{align*}
where the second inequality follows from using Jensen's inequality.
\end{proof}

\subsection{Proof of Lemma \ref{lem:consensus} }\label{subsec:proof_lem_consensus}
\begin{proof}
For convenience let $\Wbf = \Ibf - 1/n\1\1^T$ and $\Ybf(k)$ be defined as
\[
\Ybf(k) = \Xbf(k)-\1\xbarbf(k)^T = \Wbf\Xbf(k).
\]
Using $\Abf\1 = \1$ and Eqs. \eqref{prelim:X} and \eqref{prelim:xbar} we consider
\begin{align*}
\Ybf(k+1) &= \Xbf(k+1) - \1\xbarbf(k+1)^T\\  &= \Abf\Xbf(k) + (\Ibf-\Abf)(\Xbf(k)-\Qbf(k)) -\alpha(k)\Gbf(\Xbf(k)) - \Xibf(\Vbf(k))\\
&\qquad - \1\xbarbf(k)^T +\frac{\alpha(k)}{n}\1\sum_{i=1}^n\gbf_i(\xbf_i(k))^T + \1\bar{\xibf}(k)^T\notag\\
&= \Abf\Wbf\Xbf(k) + (\Ibf-\Abf)\Delta(k)-\alpha(k)\Wbf\Gbf(\Xbf(k)) - \Wbf\Xi(\Vbf(k)),
\end{align*}
which by taking the Frobenius norm on both sides yields
\begin{align}
\left\|\Ybf(k+1)\right\| &= \left\|\Abf\Wbf\Xbf(k) + (\Ibf-\Abf)\Delta(k)-\alpha(k)\Wbf\Gbf(\Xbf(k)) - \Wbf\Xibf(\Vbf(k))\right\|\notag\\
&\leq \left\|\Abf\Wbf\Xbf(k)\right\| + \left\|(\Ibf-\Abf)\Delta(k)\right\| + \left\|\alpha(k)\Wbf\Gbf(\Xbf(k)) - \Wbf\Xibf(\Vbf(k))\right\|\notag\\
&\leq \sigma_2\left\|\Wbf\Xbf(k)\right\| + 2\left\|\Delta(k)\right\| + \alpha(k)\left\|\Gbf(\Xbf(k))\right\| + \left\|\Xibf(\Vbf(k))\right\|,\label{lem_consensus:Eq1}
\end{align}
where the last inequality is because the largest singular values of  $\Wbf$ and $\Abf$ are smaller than $1$ and  using the Courant-Fisher theorem \cite{HJ1985}, i.e.,
\begin{align*}
\left\|\Abf\Wbf\Xbf(k)\right\| = \left\|\Abf\left(\Ibf-\frac{1}{n}\1\1^T\right)\Xbf(k)\right\|\leq \sigma_2\left\|\left(\Ibf-\frac{1}{n}\1\1^T\right)\Xbf(k)\right\| = \sigma_2\left\|\Wbf\Xbf(k)\right\|.
\end{align*}
First, using $L = \sum_{i=1}^n L_i$  Eq.\ \eqref{prop:bounded_subg_ineq} gives
\begin{align}
\|\Gbf(\Xbf(k))\| \leq \sqrt{\sum_{i=1}^n\|\gbf_i(\xbf_i(k))\|^2} \leq \sqrt{\sum_{i=1}^n L_i^2} \leq L.\label{lem_consensus:Eq1a}    
\end{align}
Second, Eq.\ \eqref{lem_project:Ineq2} yields
\begin{align}
\|\Xibf(\Vbf(k))\| \leq 2\sqrt{2}\|\Delta(k)\| + \sqrt{2}L\alpha(k).\label{lem_consensus:Eq1b}    
\end{align}
Thus using Eqs.\ \eqref{lem_consensus:Eq1a} and  \eqref{lem_consensus:Eq1b} into Eq.\ \eqref{lem_consensus:Eq1} yields Eq.\ \eqref{lem_consensus:consensus_bound}, i.e.,
\begin{align*}
\left\|\Ybf(k+1)\right\| &\leq \sigma_2\|\Ybf(k)\| + 6\|\Delta(k)\| +3L\alpha(k)\\
&\leq \sigma_2^{k+1}\|\Ybf(0)\| + 6\sum_{t=0}^{k}\sigma_2^{k-t}\|\Delta(t)\| + 3L\sum_{t=0}^{k}\sigma_2^{k-t}\alpha(t)\\
&= 6\sum_{t=0}^{k}\sigma_2^{k-t}\|\Delta(t)\| + 3L\sum_{t=0}^{k}\sigma_2^{k-t}\alpha(t),
\end{align*}
where the last equality is due to $\xbf_i(0) = \xbf_j(0)$ for all $i,j\in\Vcal$, implying $\Ybf(0) = \mathbf{0}$. 
\end{proof}

\subsection{Proof of Lemma \ref{lem:quantized_error}}
\begin{proof}
For convenience, just in this proof, we define $\wbf_i(k)$ as \[
\wbf_i(k) =  \left[ \sum_{j\in\Ncal_i}a_{ij}\qbf_j(k)\right]_{\Xcal} = \sum_{j\in\Ncal_i}a_{ij}\qbf_j(k)-\pbf_i(k), 
\]
where $\pbf_i(k) \triangleq  \sum_{j\in\Ncal_i}a_{ij}\qbf_j(k) - \wbf_i(k)$. Using the definition of the projection and since $\sum_{j\in\Ncal_i}a_{ij}\xbf_j(k)\in\Xcal$ we have
\begin{align}
\|\,\pbf_i(k)\,\| &= \left\| \sum_{j\in\Ncal_i}a_{ij}\qbf_j(k)-\wbf_i(k)\right\|\leq \left\|\sum_{j\in\Ncal_i}a_{ij}\xbf_j(k) - \sum_{j\in\Ncal_i}a_{ij}\qbf_j(k)\right\|\notag\\
&\leq \sum_{j\in\Ncal_i}a_{ij}\left\|\Delta_j(k)\right\|\label{lem_quantized_error:Eq0a},
\end{align}
which implies
\begin{align}
\left\|\pbf_i(k) - \bar{\pbf}(k)\right\| &\leq  \left\|\Pbf(k) - \1\bar{\pbf}(k)^T\right\|\leq \|\Pbf(k)\|\notag\\
& \leq \sqrt{\sum_{i=1}^n\left(\sum_{j\in\Ncal_i}a_{ij}\left\|\Delta_j(k)\right\|\right)^2}\leq \|\Delta(k)\|.
\label{lem_quantized_error:Eq0b}    
\end{align}
Moreover, since $\Delta_i(k) = \xbf_i(k) - \qbf_i(k)$ we have
\begin{align}
\|\qbf_i(k) - \qbarbf(k)\|\leq \|\Qbf(k)-\1\qbarbf(k)^T\|\leq \|\Xbf(k)-\1\xbarbf(k)^T \| + \|\Delta(k)\|.    \label{lem_quantized_error:Eq0c}   
\end{align}
% which implies that 
% \begin{align}
% \|\Pbf(k)\|^2 &= \sum_{i=1}^n\|\pbf_i(k)\|^2 \leq\sum_{i=1}^n\left\|\sum_{j\in\Ncal_i}a_{ij}\xbf_j(k) - \sum_{j\in\Ncal_i}a_{ij}\qbf_j(k)\right\|^2\notag\\
% &\leq \sum_{j=1}^n \|\Delta_j(k)\|^2 = \|\Delta(k)\|^2.\label{lem_quantized_error:Eq0}   
% \end{align}
Using the triangle inequality we now consider
\begin{align}
\|\xbf_i(k+1) - \qbf_i(k)\| & \leq\|\xbf_i(k+1) - \wbf_i(k)\| + \|\wbf_i(k) - \bar{\wbf}(k)\| + \|\bar{\wbf}(k) - \qbf_i(k)\|.\label{lem_quantized_error:Eq1} 
\end{align}
We now provide an upper bound for each term on the right-hand side of Eq.\ \eqref{lem_quantized_error:Eq1}. First, by the nonexpansiveness of the projection, and Eqs.\ \eqref{prop:bounded_subg_ineq} and \eqref{distributed:QDSG}  we have
\begin{align}
\|\xbf_i(k+1)-\wbf_i(k)\|\leq \|\Delta_i(k) - \alpha(k)\gbf_i(\xbf_i(k))\| \leq \|\Delta_i(k)\| + L_i\alpha(k).\label{lem_quantized_error:Eq1a}
\end{align}
Second, using Eqs.\ \eqref{lem_quantized_error:Eq0b} and \eqref{lem_quantized_error:Eq0c} we have
\begin{align}
\left\|\wbf_i(k) - \bar{\wbf}(k)\right\| &= \left\|\sum_{j\in\Ncal_i}a_{ij}\qbf_j(k) - \bar{\qbf}(k) -  \pbf_i(k) + \bar{\pbf}(k)\right\|\notag\\
&\leq  \sum_{j\in\Ncal_i}a_{ij}\|\qbf_j(k) - \bar{\qbf}(k)\| +  \left\|\pbf_i(k) - \bar{\pbf}(k)\right\|\notag\\
&\leq  \|\Xbf(k)-\1\xbarbf(k)^T \| + 2\|\Delta(k)\|.\label{lem_quantized_error:Eq1b}
\end{align}
Third, using Eq.\ \eqref{lem_quantized_error:Eq0c} we consider 
\begin{align}
\|\bar{\wbf}(k)-\qbf_i(k)\| &= \|\bar{\qbf}(k)-\qbf_i(k) +  \bar{\pbf}(k)\|\leq \|\qbf_i(k) - \bar{\qbf}(k)\| +  \|\bar{\pbf}(k)\|\notag\\
&\leq \|\Xbf(k)-\1\xbarbf(k)^T \| + 2\|\Delta(k)\|. \label{lem_quantized_error:Eq1c}
\end{align}
Substituting Eqs.\ \eqref{lem_quantized_error:Eq1a}--\eqref{lem_quantized_error:Eq1c} into Eq.\ \eqref{lem_quantized_error:Eq1} yields
\begin{align*}
\|\xbf_i(k+1) - \qbf_i(k)\| &\leq    2\|\Xbf(k)-\1\xbarbf(k)^T\| + 5\|\Delta(k)\| +  L\alpha(k),
\end{align*}
which by Eq.\ \eqref{lem_consensus:consensus_bound} gives
\begin{align}
\|\xbf_i(k+1) - \qbf_i(k)\| &\leq 12\sum_{t=0}^{k-1}\sigma_2^{k-1-t}\|\Delta(t)\| + 6L\sum_{t=0}^{k-1}\sigma_2^{k-1-t}\alpha(t)\notag\\
&\qquad +  5\|\Delta(k)\| + L\alpha(k)\notag\\
&\leq  12\sum_{t=0}^{k}\sigma_2^{k-t}\|\Delta(t)\| + 6L\sum_{t=0}^{k}\sigma_2^{k-t}\alpha(t)\label{lem_quantized_error:Eq2}. 
\end{align}
Recall that $\gamma = 48(2+ L)\,/\,(1-\sigma_2)$ and 
\[
\Rcal_i(k+1) \triangleq \left[\qbf_i(k) - \frac{\gamma}{2}\alpha\left(k\right)\1,\,\qbf_i(k) +\frac{\gamma}{2}\alpha\left(k\right)\1\right].
\]
We now show that $\xbf_i(k+1)\in\Rcal_i(k+1)$ by induction. First, when $k=0$ we have $\|\Delta_i(0)\| = 0$ since $\xbf_i(0) = \qbf_i(0)$ for all $i\in\Vcal$. By definition, we have  $\xbf_i(0)\in\Rcal_i(0)$. Suppose it is true for some $k>0$, that is, $\xbf_i(k)\in\Rcal_i(k)$. We now show that $\xbf_i(k+1)\in\Rcal_i(k+1)$. Indeed, since $\alpha(k)$ is nonincreasing, by the definition of $\Rcal_i(k)$ we have $\Delta_i(k)$ is nonincreasing and 
\[
\Delta_i(t) \leq \frac{\gamma\alpha\left(t\right)}{2^b-1}\1,\qquad \forall i\in\Vcal,\;\;t\in[0,k]. 
\]
Using Assumption \ref{assump:bits}, i.e.,  $\sqrt{nd}\gamma/ (2^b-1) \leq 1$, we have $\|\Delta(t)\| \leq \alpha(t)$ for all $t\in[0,k]$. Thus, by using $\alpha(k)\leq\alpha(t)\leq \alpha(0) = 1$ for $t\in[0,k]$,
Eq.\ \eqref{lem_quantized_error:Eq2} gives
\begin{align*}
\|\xbf_i(k+1) - \qbf_i(k)\| &\leq 12\sum_{t=0}^{k}\sigma_2^{k-t}\alpha(t) + 6L\sum_{t=0}^{k}\sigma_2^{k-t}\alpha(t)\nonumber\\
&= 6(2+ L)\left(\sum_{t=0}^{\lfloor k/2\rfloor}\sigma_2^{k-t}\alpha(t)+\sum_{t=\lceil k/2\rceil }^k \sigma_2^{k-t}\alpha(t)\right)\notag\\
&\leq 6(2+ L)\left(\frac{\alpha(0)\sigma_2^{\lceil k/2\rceil}}{1-\sigma_2}+\frac{\alpha(\lceil k/2\rceil)}{1-\sigma_2}\right)\notag\\
&\leq \frac{24(2+ L)}{1-\sigma_2}\alpha(k) = \frac{\gamma}{2}\alpha(k),
\end{align*}
where in the last inequality is due to $\sigma_{2}^{k}\leq \alpha(k),$ $\alpha(0) = 1$, $\alpha(\lceil k/2\rceil)\leq 2\alpha(k)$ since we only consider $\alpha(k) = 1/(k+1)$ or $\alpha(k) = 1/\sqrt{k+1}$. This concludes our proof.   
\end{proof}

\subsection{Proof of Lemma \ref{lem:consensus_bound}}

\begin{proof}
First, Eq.\ \eqref{lem_quantized_error:error_bound} yields 
\begin{align*}
\|\Delta(k)\| \leq \frac{\sqrt{nd}\gamma}{2^b-1}\alpha(k), 
\end{align*}
which using Eq.\ \eqref{lem_consensus:consensus_bound} gives
\begin{align*}
\|\Xbf(k+1)-\1\xbarbf(k+1)^T\| &\leq \frac{6\sqrt{nd}\gamma}{2^b-1}\sum_{t=0}^{k}\sigma_2^{k-t}\alpha(t) + 3L\sum_{t=0}^{k}\sigma_2^{k-t}\alpha(t)\notag\\
&\leq \left(\frac{6\sqrt{nd}\gamma+3L2^b}{2^b-1}\right)\left(\sum_{t=0}^{\lfloor k/2\rfloor }\sigma_2^{k-t}\alpha(t) + \sum_{t=\lceil k/2\rceil}^{k}\sigma_2^{k-t}\alpha(t)\right)\nonumber\\
&\leq \left(\frac{6\sqrt{nd}\gamma+3L2^b}{2^b-1}\right)\left(\sum_{t=0}^{\lfloor k/2\rfloor }\sigma_2^{k-t} + \alpha(\lceil k/2\rceil) \sum_{t=\lfloor k/2\rfloor}^{k}\sigma_2^{k-t}\right)\notag\\
&\leq \left(\frac{6\sqrt{nd}\gamma+3L2^b}{2^b-1}\right)\left(\frac{1}{1-\sigma_2}\sigma_2^{\lceil k/2\rceil} + \frac{1}{1-\sigma_2}\alpha(\lceil k/2\rceil)\right),
\end{align*}
which since $\lim_{k\rightarrow\infty}\alpha(k) = 0$ gives Eq.\ \eqref{lem_consensus:asymp_conv}.

Suppose now that the condition \eqref{lem_consensus:square_summable} is held. Then, for some $K\geq0$ we have Eq.\ \eqref{lem_consensus:finite_sum}, i.e., 
\begin{align*}
&\sum_{k=0}^{K}\alpha(k)\|\Xbf(k)-\1\xbarbf(k)^T\|\leq \left(\frac{6\sqrt{nd}\gamma+3L2^b}{2^b-1}\right)\sum_{k=0}^{K}\alpha(k)\sum_{t=0}^{k-1}\sigma_2^{k-1-t}\alpha(t)\notag\\
&\leq \left(\frac{6\sqrt{nd}\gamma+3L2^b}{2^b-1}\right)\sum_{k=0}^{K}\sum_{t=0}^{k-1}\sigma_2^{k-1-t}\alpha^2(t)= \left(\frac{6\sqrt{nd}\gamma+3L2^b}{2^b-1}\right)\sum_{t=0}^{K}\alpha^2(t)\sum_{k=t+1}^{K}\sigma_2^{k}\notag\\
&\leq  \left(\frac{6\sqrt{nd}\gamma+3L2^b}{(1-\sigma_2)(2^b-1)}\right)\sum_{t=0}^{K}\alpha^2(t) \stackrel{\eqref{lem_consensus:square_summable}}{<} \infty. 
\end{align*}
\item Suppose now that $\alpha(k) = 1/\sqrt{k+1}$. Then by the inequality above we have Eq.\ \eqref{lem_consensus:rate}, i.e., 
\begin{align*}
\sum_{k=0}^{K}\alpha(k)\|\Xbf(k)-\1\xbarbf(k)^T\|
&\leq \left(\frac{6\sqrt{nd}\gamma+3L2^b}{(1-\sigma_2)(2^b-1)}\right)\sum_{t=0}^{K}\frac{1}{t+1}\notag\\
&\leq \left(\frac{6\sqrt{nd}\gamma+3L2^b}{(1-\sigma_2)(2^b-1)}\right)(1+\ln(K+1)),
\end{align*}
where we use the integral test in the last inequality to have
\begin{align*}
\sum_{t=0}^{K-1}\frac{1}{t+1}\leq 1 + \int_{0}^{K}\frac{1}{t+1}dt \leq 1 + \ln(K+1). 
\end{align*}
\end{proof}

\subsection{Proof of Lemma \ref{lem:opt_dist}}\label{subsec:proof_lem_opt_dist}

\begin{proof}
Let $\xbf^*$ be a solution of problem \eqref{prob:obj}. For convenience, let $\rbf(k) = \xbarbf(k)-\xbf^*.$ First, recall from Eq. \eqref{lem_project:Ineq2} that
\begin{align}
\|\Xibf(\Vbf(k))\|^2 \leq 8\|\Delta(k)\|^2 + 2L^2\alpha^2(k).\label{lem_opt_dist:Eq0a}     
\end{align}
Second, by the definition of the projection we have
\begin{align}
\xibf_i(\vbf_i(k))^T(\xbf^* - \vbf_i(k))\leq -\|\xibf_i(\vbf_i(k))\|^2.    \label{lem_opt_dist:Eq0b}
\end{align}
We now use Eq.\ \eqref{prelim:xbar} to have
\begin{align}
\|\rbf(k+1)\|^2 &= \left\|\xbarbf(k)-\xbf^*-\frac{\alpha(k)}{n}\sum_{i=1}^n\gbf_i(\xbf_i(k)) -\xibarbf(k)\right\|^2\notag\\
&= \|\rbf(k)\|^2 -2(\xbarbf(k)-\xbf^*)^T\xibarbf(k) -\frac{2\alpha(k)}{n}(\xbarbf(k)-\xbf^*)^T\sum_{i=1}^n \gbf_i(\xbf_i(k))\notag\\
&\qquad + \left\|\frac{\alpha(k)}{n}\sum_{i=1}^n\gbf_i(\xbf_i(k)) + \xibarbf(k)\right\|^2\notag\\
&\leq \|\rbf(k)\|^2 -2(\xbarbf(k)-\xbf^*)^T\xibarbf(k) -\frac{2\alpha(k)}{n}(\xbarbf(k)-\xbf^*)^T\sum_{i=1}^n \gbf_i(\xbf_i(k))\notag\\
&\qquad + 2\left\|\frac{\alpha(k)}{n}\sum_{i=1}^n\gbf_i(\xbf_i(k))\right\|^2 + 2\left\|\xibarbf(k)\right\|^2\notag\\
&\stackrel{\eqref{lem_opt_dist:Eq0a}}{\leq}  \|\rbf(k)\|^2 -2(\xbarbf(k)-\xbf^*)^T\xibarbf(k) -\frac{2\alpha(k)}{n}(\xbarbf(k)-\xbf^*)^T\sum_{i=1}^n \gbf_i(\xbf_i(k))\notag\\
&\qquad + \frac{16\Delta^2(k)+6L^2\alpha^2(k)}{n}\cdot\label{lem_opt_dist:Eq1}
\end{align}
We now analyze the second term on the right-hand side of Eqs.\ \eqref{lem_opt_dist:Eq1} by using Eqs.\ \eqref{lem_opt_dist:Eq0a} and \eqref{lem_opt_dist:Eq0b}
\begin{align}
&-2(\xbarbf(k)-\xbf^*)^T\xibarbf(k) = -\frac{2}{n}\sum_{i=1}^n\xibf_i(\vbf_i(k))^T(\xbarbf(k)-\xbf^*)\notag\\
&\qquad= -\frac{2}{n}\sum_{i=1}^n\xibf_i(\vbf_i(k))^T(\xbarbf(k)-\vbf_i(k)+\vbf_i(k)-\xbf^*)\notag\\
&\qquad\leq \frac{2}{n}\sum_{i=1}^n\|\xibf_i(\vbf_i(k))\|\,\|\xbarbf(k)-\vbf_i(k)\| - \frac{2}{n}\sum_{i=1}^n\xibf_i(\vbf_i(k))^T(\vbf_i(k)-\xbf^*)\notag\\
&\qquad\stackrel{\eqref{lem_opt_dist:Eq0a}}{\underset{\eqref{lem_opt_dist:Eq0b}}{\leq}} \frac{2\left(2\sqrt{2}\|\Delta(k)\|+\sqrt{2}L\alpha(k)\right)}{n}\sum_{i=1}^n\,\|\xbarbf(k)-\vbf_i(k)\| -\frac{2}{n}\sum_{i=1}^n\|\xibf_i(\vbf_i(k))\|^2\notag\\
&\qquad\leq \frac{2\left(2\sqrt{2}\|\Delta(k)\|+\sqrt{2}L\alpha(k)\right)}{\sqrt{n}}\|\Xbf(k)-\1\xbarbf(k)^T\|\notag\\
&\qquad \qquad + \frac{2\left(2\sqrt{2}\|\Delta(k)\|+\sqrt{2}L\alpha(k)\right)(2\sqrt{n}\|\Delta(k)\| + L\alpha(k))}{n}
,\label{lem_opt_dist:Eq1a}
\end{align}
where the last inequality is due to
\begin{align*}
&\sum_{i=1}^n\|\xbarbf(k)-\vbf_i(k)\|\notag\\
&\leq \sum_{i=1}^n\left\|\xbarbf(k)-\sum_{j=1}^na_{ij}\xbf_j(k) + \Delta_i(k) -\sum_{j\in\Ncal_i}a_{ij}\Delta_j(k) - \alpha(k)\gbf_i(\xbf_i(k))\right\|\notag\\
&\leq \sqrt{n}\|\Xbf(k)-\1\xbarbf(k)^T\| + 2\sqrt{n}\|\Delta(k)\| + L\alpha(k),
\end{align*}
which uses the Jensen's inequality. 
Next, we analyze the third term on the right-hand side of Eq.\ \eqref{lem_opt_dist:Eq1}
\begin{align}
&-\frac{2\alpha(k)}{n}(\xbarbf(k)-\xbf^*)^T\sum_{i=1}^n \gbf_i(\xbf_i(k))\notag\\
&\qquad = -\frac{2\alpha(k)}{n}\sum_{i=1}^n\gbf_i(x_i(k))^T(\xbarbf(k)-\xbf_i(k)) -\frac{2\alpha(k)}{n}\sum_{i=1}^n\gbf_i(\xbf_i(k))^T(\xbf_i(k)-\xbf^*)\notag\\
&\qquad \leq \frac{2L\alpha(k)}{n}\|\Xbf(k)-\1\xbarbf(k)^T\| -\frac{2\alpha(k)}{n}\sum_{i=1}^n\gbf_i(\xbf_i(k))^{T}(\xbf_i(k)-\xbf^*).\label{lem_opt_dist:Eq1b}
\end{align}
Substituting Eqs.\ \eqref{lem_opt_dist:Eq1a} and \eqref{lem_opt_dist:Eq1b} into \eqref{lem_opt_dist:Eq1} we obtain \begin{align*}
\|\rbf(k+1)\|^2 
&\leq  \|\rbf(k)\|^2  -\frac{2\alpha(k)}{n}\sum_{i=1}^n\gbf_i(\xbf_i(k))^{T}(\xbf_i(k)-\xbf^*) \notag\\
&\qquad + \frac{2\left(2\sqrt{2}\|\Delta(k)\|+\sqrt{2}L\alpha(k)\right)}{\sqrt{n}}\|\Xbf(k)-\1\xbarbf(k)^T\|\notag\\ 
&\qquad +  \frac{2\left(2\sqrt{2}\|\Delta(k)\|+\sqrt{2}L\alpha(k)\right)(2\sqrt{n}\|\Delta(k)\| + L\alpha(k))}{n}
\notag\\
&\qquad  + \frac{2L\alpha(k)}{n}\|\Xbf(k)-\1\xbarbf(k)^T\| +  \frac{16\Delta^2(k)+6L^2\alpha^2(k)}{n},
\end{align*}
which gives us Eq.\ \eqref{lem_opt_dist:Ineq}, i.e.,

\begin{align*}
\|\rbf(k+1)\|^2 &\leq \|\rbf(k)\|^2-\frac{2\alpha(k)}{n}\sum_{i=1}^n\gbf_i(\xbf_i(k))^{T}(\xbf_i(k)-\xbf^*) \notag\\
&\qquad  + \frac{2\left(4\|\Delta(k)\|+3L\alpha(k)\right)}{\sqrt{n}}\|\Xbf(k)-\1\xbarbf(k)^T\|\notag\\ 
&\qquad + \frac{32\|\Delta(k)\|^2 + 10L^2\alpha^2(k)+12L\alpha(k)\|\Delta(k)\|}{\sqrt{n}}\notag\\
&\leq \|\rbf(k)\|^2-\frac{2\alpha(k)}{n}\sum_{i=1}^n\gbf_i(\xbf_i(k))^{T}(\xbf_i(k)-\xbf^*) \notag\\
&\qquad  + \frac{2\left(4\sqrt{nd}\gamma+3L2^{b}\right)}{\sqrt{n}(2^b-1)}\alpha(k)\|\Xbf(k)-\1\xbarbf(k)^T\|\notag\\ 
&\qquad + \frac{2(4\sqrt{nd}\gamma + 3L2^{b})^2}{(2^b-1)^2\sqrt{n}}\alpha^2(k),
\end{align*}
where the last inequality we use Eq.\  \eqref{lem_quantized_error:error_bound} to have
\begin{align*}
\|\Delta(k)\| \leq \frac{\sqrt{nd}\gamma}{2^b-1}\alpha(k).
\end{align*}
\end{proof}